\documentclass[final]{siamart220329}

\usepackage{amsmath,amsfonts,stmaryrd,enumitem}
\usepackage{booktabs}
\usepackage{multirow}
\usepackage{algorithmic}
\ifpdf
\DeclareGraphicsExtensions{.eps,.pdf,.png,.jpg}
\else
\DeclareGraphicsExtensions{.eps}
\fi

\newsiamremark{remark}{Remark}
\newsiamremark{hypothesis}{Hypothesis}
\crefname{hypothesis}{Hypothesis}{Hypotheses}

\headers{Accelerated DCA for SEiCP and SQEiCP}{Yi-Shuai Niu}

\title{An Accelerated DC Programming Approach with Exact Line Search for The Symmetric Eigenvalue Complementarity Problem\thanks{Submitted to the editors DATE.
		\funding{This work was funded by the Natural Science Foundation of China (Grant No: 11601327).}}}

\author{Yi-Shuai Niu\thanks{Department of Applied Mathematics, The Hong Kong Polytechnic University, Hong Kong 
		(\email{niuyishuai82@hotmail.com}).}
}

\DeclareMathOperator{\diag}{diag}
\DeclareMathOperator{\R}{\mathbb{R}}

\DeclareMathOperator{\setC}{\mathcal{C}}
\DeclareMathOperator{\setA}{\mathcal{A}}
\DeclareMathOperator{\argmin}{argmin}
\DeclareMathOperator{\dom}{dom}
\DeclareMathOperator{\PD}{PD}
\DeclareMathOperator{\SPD}{SPD}
\DeclareMathOperator{\SC}{SC}
\DeclareMathOperator{\prox}{prox}

\newcommand{\SEiCP}[1]{\text{SEiCP}(#1)}

\newcommand{\SQEiCP}[1]{\text{SQEiCP}(#1)}

\ifpdf
\hypersetup{
	pdftitle={Accelerated DCA for SEiCP},
	pdfauthor={Y.S. Niu}
}
\fi

\begin{document}
	
	\maketitle
	
	\begin{abstract}
		In this paper, we are interested in developing an accelerated Difference-of-Convex (DC) programming algorithm based on the exact line search for efficiently solving the Symmetric Eigenvalue Complementarity Problem (SEiCP) and Symmetric Quadratic Eigenvalue Complementarity Problem (SQEiCP). We first proved that any SEiCP is equivalent to SEiCP with symmetric positive definite matrices only. Then, we established DC programming formulations for two equivalent formulations of SEiCP (namely, the logarithmic formulation and the quadratic formulation), and proposed the accelerated DC algorithm (BDCA) by combining the classical DCA with inexpensive exact line search by finding real roots of a binomial for acceleration. We demonstrated the equivalence between SQEiCP and SEiCP, and extended BDCA to SQEiCP. Numerical simulations of the proposed BDCA and DCA against KNITRO, FILTERED and MATLAB FMINCON for SEiCP and SQEiCP on both synthetic datasets and Matrix Market NEP Repository are reported. BDCA demonstrated dramatic acceleration to the convergence of DCA to get better numerical solutions, and outperformed KNITRO, FILTERED, and FMINCON solvers in terms of the average CPU time and average solution precision, especially for large-scale cases.
	\end{abstract}
	
	
	
	\section{Introduction} \label{sec:Intro}
	Symmetric Eigenvalue Complementarity Problem (SEiCP) consists of finding complementary eigenvectors $x\in \R^n\setminus\{0\}$ and complementary eigenvalues $\lambda\in \R$ such that 
	\begin{equation}\label{eq:seicp}
		\begin{cases}
			w = \lambda B   x - A  x,\\
			x^{\top}  w = 0,\\
			0\neq x\geq 0, w\geq 0,
		\end{cases}\tag{SEiCP}
	\end{equation}
	where $x^{\top}$ is the transpose of $x$, $A$ is a real symmetric $\R^{n\times n}$ matrix, and $B$ is a real \emph{symmetric positive definite} ($\SPD$) matrix. 
	The SEiCP appeared in the study of static equilibrium states of mechanical systems with unilateral friction in \cite{Costa01}, and found many applications in engineering \cite{Costa04,Seeger99}.  
	
	Concerning the feasibility of \eqref{eq:seicp}, it is known that \eqref{eq:seicp} always has a solution \cite{Judice09}. The existence of solutions is even guaranteed under the weaker hypothesis that $B$ is \emph{strictly copositive} ($\SC$), i.e., $x^{\top} B x >0, \forall 0\neq x\geq 0$. \eqref{eq:seicp} has a positive complementary eigenvalue if and only if there exists some $x\geq 0$ such that $x^{\top}   A  x > 0$ \cite{Queiroz03}. 
	For example, when $A$ is a $\SC$ matrix, then \eqref{eq:seicp} has a positive complementary eigenvalue. In general, \eqref{eq:seicp} has at most $2^n-1$ positive $\lambda$-solutions \cite{Queiroz03}.  
	
	An important extension of \eqref{eq:seicp} is called \emph{Symmetric Quadratic Eigenvalue Complementarity Problem} (SQEiCP) introduced in \cite{Seeger11}, where some applications are highlighted. SQEiCP consists of finding quadratic complementary eigenvectors $x\in \R^n\setminus\{0\}$ and quadratic complementary eigenvalues $\lambda\in \R$ such that 
	\begin{equation}\label{eq:sqeicp}
		\begin{cases}
			w = \lambda^2 A   x  + \lambda B   x + C   x,\\
			x^{\top}  w = 0,\\
			0\neq x\geq 0, w\geq 0,
		\end{cases}\tag{SQEiCP}
	\end{equation}
	where $A$, $B$ and $C$ are $n\times n$ real symmetric matrices.
	
	Concerning the feasibility of \eqref{eq:sqeicp}, as opposed to \eqref{eq:seicp}, the \eqref{eq:sqeicp} may have no solution even when the leading matrix $A$ is $\SPD$. It is known that \eqref{eq:sqeicp} is feasible if the co-regular (i.e., $x^{\top}Ax\neq 0, \forall 0\neq x\geq 0$) and co-hyperbolic (i.e., $(x^{\top}Bx)^2\geq 4(x^{\top}Ax)(x^{\top}Cx), \forall 0\neq x\geq 0$) conditions are satisfied \cite{Seeger11}. Note that these two conditions are not necessary for the existence of a solution to \eqref{eq:sqeicp}, it is shown in \cite{Fernandes} that \eqref{eq:sqeicp} has a solution if $C$ is symmetric SC, $B=0$ and there exists a vector $x\geq 0$ such that $x^{\top}Ax<0$. 
	
	In our paper, we first demonstrate that any \eqref{eq:seicp} is equivalent to \eqref{eq:seicp} with $\SPD$ matrix $A+\mu B$ for some large enough $\mu$. Then, we propose applying an accelerated Difference-of-Convex (DC) programming approach for solving two renowned equivalent formulations for \eqref{eq:seicp} with $\SPD$ matrices $A$ and $B$ (namely, the logarithmic formulation and the quadratic formulation). The accelerated DC algorithm is called Boosted-DCA (cf. BDCA) established in our recent paper \cite{Niu19higher} based on the classical DCA with line search for convex constrained DC programs. BDCA applied to the logarithmic formulation of \eqref{eq:seicp} requires solving a sequence of convex subproblems involving a strongly convex objective function over a simplex, which can be efficiently solved by the proposed FISTA algorithm \cite{beck2009fista}, where a sequence of simplex projections are computed with $O(n\log(n))$ worst case complexity. Whereas BDCA applied to the quadratic formulation of \eqref{eq:seicp} requires solving a sequence of linear minimization problems over an ellipsoid and the nonnegative orthant, which can be solved by invoking quadratic programming solvers such as MOSEK, GUROBI and CPLEX. Moreover, we show that the exact line search in BDCA can be computed inexpensively by finding real roots of a binomial. Concerning the extension to \eqref{eq:sqeicp}, we propose an equivalent \eqref{eq:seicp} formulation for finding any positive and negative quadratic complementary eigenvalue of \eqref{eq:sqeicp}. Hence, the proposed BDCA algorithms for \eqref{eq:seicp} can be naturally extended for \eqref{eq:sqeicp}. 
	
	The paper is organized as follows: In \Cref{sec:seicp_formulations}, we demonstrate that any \eqref{eq:seicp} is equivalent to \eqref{eq:seicp} with $\SPD$ matrices only, and three equivalent \eqref{eq:seicp} formulations (namely, \eqref{prob:eicp_frac}, \eqref{prob:seicp_log} and \eqref{prob:seicp_qcqp}) are introduced. After a brief summary of some fundamentals in DC programming, DCA and BDCA algorithms in \Cref{sec:DC}, we focus in \Cref{sec:dcforeicp} on developing DC formulations and DCA/BDCA algorithms for \eqref{prob:seicp_log} and \eqref{prob:seicp_qcqp} models of \eqref{eq:seicp}. These approaches are extended to \eqref{eq:sqeicp} in \Cref{sec:SQEiCP} where the equivalent formulation of \eqref{eq:sqeicp} as two \eqref{eq:seicp} are established. Numerical simulations of the proposed BDCA and DCA algorithms against KNITRO, FILTERSD and MATLAB FMINCON solvers, tested on both synthetic datasets and Matrix Market NEP Repository for \eqref{eq:seicp} and \eqref{eq:sqeicp}, are reported in \Cref{sec:Simulations}. Some concluding remarks and important future research topics are summarized in the last section.       
	
	\section{SEiCP Formulations}\label{sec:seicp_formulations}
	It is not difficult to see that \eqref{eq:seicp} is equivalent to 
	\begin{equation}\label{eq:seicp1}
		\begin{cases}
			w = \lambda B   x - A  x,\\
			x^{\top}  w = 0,\\
			e^{\top}  x = 1,\\
			x\geq 0, w\geq 0,
		\end{cases}
	\end{equation}
	by introducing a so-called \emph{regularity constraint} $e^{\top}   x=1$ where $e$ denotes the vector of ones. This constraint helps to eliminate $x=0$. Then, replacing $w$ by $\lambda B   x - A  x$, problem \eqref{eq:seicp1} turns to 
	\begin{equation}\label{eq:seicp2}
		\begin{cases}
			\lambda x^{\top}   B   x - x^{\top}   A  x = 0,\\
			\lambda B   x - A  x \geq 0,\\
			e^{\top}  x = 1,\\
			x\geq 0.
		\end{cases}
	\end{equation}
	Now, let us denote the solution set of problem \eqref{eq:seicp2} by $\SEiCP{A,B}$ and $\Omega := \{x\in \R^n: e^{\top}   x = 1, x\geq 0\}$ be the unit simplex. Clearly, for any solution $(x,\lambda)\in \SEiCP{A,B}$, we have $x^{\top}   B   x > 0$ (since $B\in \SPD$) and 
	\begin{equation}\label{eq:thm1.1}
		\lambda = \frac{x^{\top}   A   x}{x^{\top}   B   x},
	\end{equation} 
	where the \emph{Rayleigh quotient} \eqref{eq:thm1.1} is well defined and derived from $\lambda x^{\top}   B   x - x^{\top}   A  x = 0$ by dividing the nonzero term $x^{\top}   B   x$ on both sides. The next theorem shows that any $\SEiCP{A,B}$ is equivalent to $\SEiCP{A+\mu B,B}$ for all $\mu \in \R$.
	\begin{theorem}\label{thm:1} 
		For all $\mu \in \R$, 
			\begin{equation}\label{eq:thm1.2}
				(x,\lambda)\in \SEiCP{A,B} \Leftrightarrow (x,\lambda+\mu)\in \SEiCP{A+\mu B, B}.
			\end{equation}	
	\end{theorem}
	\begin{proof}
			For all $\mu\in \R$, we get immediately from
			$$\begin{cases}
				(\lambda + \mu) x^{\top}   B   x - x^{\top}   (A+\mu B)  x = \lambda x^{\top}   B   x - x^{\top}   A  x,\\
				(\lambda+\mu) B   x - (A+\mu B)   x = 
				\lambda B   x - A  x,
			\end{cases}$$
			the desired equivalence.
	\end{proof}

	\Cref{thm:1} indicates that any SEiCP with $A\notin \SPD$ is equivalent to an SEiCP with $A\in \SPD$, because $A+\mu B\in \SPD$ for large enough $\mu$. Note that the smallest $\mu$ can be computed by solving the semidefinite program (SDP) $\min\{ \mu : A + \mu B \succeq 0\}$, which can be numerically solved by SDP solvers such as MOSEK, SeDuMi, CSDP, DSDP, SDPT3 and SDPA. We can easily estimate an upper bound for $\mu$ without solving SDP by $ \lvert\lambda_{\min}(A)\rvert/\lambda_{\min}(B)$,
	where $\lambda_{\min}(\cdot)$ denotes the smallest eigenvalue. Therefore, without loss of generality, we suppose that both $A$ and $B$ in \eqref{eq:seicp} are $\SPD$ matrices as stated in \cref{hyp1}.
	
	\begin{hypothesis}\label{hyp1}
		$A$ and $B$ are SPD matrices for \eqref{eq:seicp}.
	\end{hypothesis}
	
	\begin{corollary}\label{cor:1}
		Under \cref{hyp1}, the problem \eqref{eq:seicp} is feasible and all $\lambda$-solutions in $\SEiCP{A,B}$ are strictly positive.
	\end{corollary}
	\begin{proof}
		The feasibility of SEiCP follows from the fact that an $\SPD$ matrix $A$ is strictly copositive. Then, we get from \eqref{eq:thm1.1} that all $\lambda$-solutions are positive since $\lambda = (x^{\top}Ax)/(x^{\top}Bx)$ where both $x^{\top}Ax$ and $x^{\top}Bx$ are strictly positive for all $x\neq 0$.
	\end{proof}
	
	\Cref{thm:1} and \Cref{cor:1} reveal an important fact that : by choosing $\mu\geq 0$ large enough such that $A+\mu B\in \SPD$, we can find $(x,\lambda)\in \SEiCP{A+\mu B,B}$ with a positive eigenvalue $\lambda$, then $(x,\lambda - \mu)\in \SEiCP{A,B}$.
	
    There are several equivalent formulations for $\SEiCP{A,B}$ as follows:
	
	\subsection{Rayleigh quotient formulation}
	An equivalent formulation of SEiCP, namely \emph{Rayleigh quotient formulation}, is given by 
	\begin{equation}\label{prob:eicp_frac}
		\max\left\{\frac{x^{\top}   A   x}{x^{\top}   B   x}: x\in \Omega \right\}. \tag{RP}
	\end{equation}
	The Rayleigh quotient $(x^{\top}   A   x)/(x^{\top}   B   x)$ is well defined on $\Omega$ since $B$ is SPD.
	
	\begin{proposition}\label{prop:RP}
		Under \Cref{hyp1}. For any stationary point $\bar{x}$ of \eqref{prob:eicp_frac}, we have  
		$$(\bar{x}, (\bar{x}^{\top}   A   \bar{x})/ (\bar{x}^{\top}   B   \bar{x}))\in \SEiCP{A,B}.$$
	\end{proposition}
	\begin{proof}
		The result is known in \cite[Proposition 9]{Queiroz03} with strictly positive Rayleigh quotient on $\Omega$ since both $A$ and $B$ are SPD matrices.  
	\end{proof}

	\subsection{Logarithmic formulation}
	Due to the positivity of the Rayleigh quotient on $\Omega$ and $A,B\in \PD$, the logarithmic metric function:
	$$L(x) = \ln\left(\frac{x^{\top}   A   x}{x^{\top}   B   x}\right) = \ln(x^{\top} A x) - \ln(x^{\top} B x)$$
	is well defined on $\Omega$, and the problem \eqref{prob:eicp_frac} is equivalent to 
	\begin{equation}\label{prob:seicp_log}
		\max\left\{ L(x): x\in \Omega \right\}. \tag{LnP}
	\end{equation}
	
	\begin{proposition}\label{prop:logformulation}
		Under \Cref{hyp1}. For any stationary point $\bar{x}$ of \eqref{prob:seicp_log}, we have $$(\bar{x}, (\bar{x}^{\top}   A   \bar{x})/ (\bar{x}^{\top}   B   \bar{x}))\in \SEiCP{A,B}.$$
	\end{proposition}
	\begin{proof}
		This is an immediate consequence of \Cref{prop:RP}.
	\end{proof}
	
	\subsection{Quadratic formulation}
	The problem \eqref{prob:eicp_frac} can be rewritten as maximizing a convex quadratic function over a compact convex set defined as 
	\begin{equation}\label{prob:seicp_qcqp}
		\max\{x^{\top}   A   x:  x^{\top}   B   x\leq 1, x\geq 0 \}. \tag{QP}
	\end{equation}
	\begin{proposition}\label{prop:QP}
		Under \Cref{hyp1}, for any nonzero stationary point $\bar{x}$ of the problem \eqref{prob:seicp_qcqp}, we have $\bar{x}^{\top}   B   \bar{x}=1$ and
		$$(\bar{x}, \bar{x}^{\top}   A   \bar{x})\in \SEiCP{A,B}.$$
	\end{proposition}
	\begin{proof}
		This is known in \cite[Theorem 2.2]{Judice08}. 
	\end{proof}	
	
	In \Cref{sec:dcforeicp}, we will represent the formulations \eqref{prob:seicp_qcqp} and \eqref{prob:seicp_log} as Difference-of-Convex (DC) programming problems and propose accelerated DC algorithms for their numerical solutions.
	
	\section{DCA and BDCA}\label{sec:DC} Let us briefly present the renowned Difference-of-Convex (DC) algorithm -- DCA and the proposed accelerated DC algorithm -- BDCA for solving the convex constrained DC program.

	The \emph{convex constrained DC program} is defined by 
	\begin{equation}
		\label{prob:dcp}
		\alpha = \min \{f(x):=g(x)-h(x): x\in \setC \}, \tag{P}
	\end{equation}
	where $\setC$ is a nonempty closed convex set in $\R^n$, the objective function $f$ is called DC if it can be written as $g-h$ where $g$ and $h$ are $\Gamma_0(\R^n)$ functions defined as the set of all proper closed and convex functions from $\R^n$ to $(-\infty,\infty]$ (classical terminologies in convex analysis, see e.g., \cite{Rockafellar}), and the optimal value $\alpha$ is supposed to be finite. This problem is equivalent to the so-called \emph{standard DC program} 
	$$\min \{ (g+\chi_{\mathcal{C}})(x) - h(x): x\in \R^n \}$$
	by introducing the indicator function of $\setC$ defined by
	$$\chi_{\setC}(x)=\begin{cases}
		0,& \text{if }x\in \mathcal{C},\\
		\infty,& \text{otherwise}.
	\end{cases}$$
	Clearly, both $g+\chi_{\mathcal{C}}$ and $h$ belong to  $\Gamma_0(\R^n)$. 
	\paragraph{\textbf{DCA}} One of the most renowned algorithm for solving \eqref{prob:dcp} is called \emph{DCA}, which is first introduced by Pham Dinh Tao in 1985 as an extension of the
	subgradient method \cite{Pham1986algorithms}, and extensively developed by 
	Le Thi Hoai An and Pham Dinh Tao since 1994 (see \cite{Pham97,Pham98,Pham05,DCA30} and the references therein).
	
	DCA consists of constructing a sequence $\{x^k\}$ by solving convex subproblems as 
	\begin{equation}\label{alg:DCA}
		\boxed{x^{k+1}\in \argmin \{ g(x) - \langle x, y^k \rangle : x\in \setC \},\quad y^k\in \partial h(x^k),} \tag{DCA}
	\end{equation}
	where 
	$\partial h(x^k)$ denotes the (convex) subdifferential of $h$ at $x^k$ defined by 
	$$
	\partial h(x^k): = \{ y\in \R^n: h(x) \geq h(x^k) + \langle x-x^k, y\rangle,\forall x\in \R^n \},
	$$ 
	which generalizes the derivative in the sense that the convex function $h$ is differentiable at $x^k$ if and only if $\partial h(x^k)$ reduces to the singleton $\{\nabla h(x^k)\}$.
	The convex subproblem required in DCA is to minimize a convex majorization (cf. surrogate) of the DC function $f$ derived by linearizing $h$ at the iterate $x^k$.
	
	DCA enjoys some convergence properties summarized in the next theorem.
	\begin{theorem}[Convergence theorem of DCA, see e.g., \cite{Pham97,niu2022convergence,Niu19higher}]\label{thm:convDCA}
		Let $\{x^k\}$ be the sequence generated by DCA for problem \eqref{prob:dcp} from $x^0\in \dom \partial h$. Suppose that both $\{x^k\}$ and $\{y^k\}$ are bounded. Then
		\begin{itemize}[leftmargin=12pt]
			\item[$\bullet$] The sequence $\{f(x^k)\}$ is decreasing and bounded from below.
			\item[$\bullet$] Every cluster point $x^*$ of the sequence $\{x^k\}$ is a \emph{DC critical point}, i.e., $\partial (g+\chi_{\setC})(x^*)\cap \partial h(x^*)\neq \emptyset$.
			\item[$\bullet$] If $h$ is continuously differentiable on $\R^n$, then every cluster point $x^*$ of the sequence $\{x^k\}$ is a \emph{strongly DC critical point}, i.e., $\nabla h(x^*) \in \partial (g+\chi_{\setC})(x^*)$.
			\item If $f$ is a KL function, either $g$ or $h$ is strongly convex, $h$ has locally Lipschitz continuous gradient over $\setC$, and $\setC$ is a semi-algebraic set (i.e., a set of polynomial equations and inequalities), then the sequence $\{x^k\}$ is convergent, whose limit point is a stationary point of \eqref{prob:dcp}, i.e., a KKT point of \eqref{prob:dcp}.
		\end{itemize}
	\end{theorem}
	Note that the last global convergence property is an immediate consequence of \cite[Theorem 5]{niu2022convergence} (see also \cite{le2018convergence}) where the function $x\mapsto f(x)+\chi_{\setC}(x)$ is a KL function satisfying the well-known Kurdyka-{\L}ajasiewicz property (see e.g., \cite[Definition 3]{niu2022convergence}), which is the key ingredient to guarantee the global convergence of the sequence $\{x^k\}$. The KL function is ubiquitous in applications, e.g., the semialgebraic, subanalytic, log and exp are KL functions (see \cite{kurdyka1998gradients,bolte2007lojasiewicz,attouch2013convergence} and the
	references therein).
	
	In practice, DCA is often terminated by one of the following conditions:
	\begin{itemize}[leftmargin=12pt]
		\item[$\bullet$] $\|x^{k+1}-x^{k}\|/(1+\|x^{k+1}\|)\leq \varepsilon_1$,
		\item[$\bullet$] $\lvert f(x^{k+1})-f(x^{k})\rvert/(1+\lvert f(x^{k+1})\rvert)\leq \varepsilon_2$,
	\end{itemize}
	for some given tolerances $\varepsilon_1>0$ and $\varepsilon_2>0$.
	
	It is worth noting that DCA has been successfully applied to solve EiCP and QEiCP in several literatures such as \cite{niu2019improved,Niu12,Niu15,Lethi12}.
	\paragraph{\textbf{BDCA}} DCA combining with line search for acceleration, namely Boosted DCA (cf. BDCA), is first proposed by Artacho et al. in 2018 \cite{BDCA_S} for unconstrained smooth DC program and extended by Niu et al. in 2019 \cite{Niu19higher} for general convex constrained smooth and nonsmooth DC programs. 
	
	Let us denote $z^k\in \argmin \{g(x) - \langle y^k,x \rangle\}$ for an optimal solution of the convex subproblem of DCA. The general idea of BDCA in \cite{Niu19higher} is to introduce a line search along the \emph{DC descent direction} (a feasible and descent direction generated by two consecutive iterates of DCA) as $d^k:= z^{k}-x^{k}$ to find a better candidate $x^{k+1}$. It is shown in \cite{Niu19higher} that $f'(z^{k};d^k)\leq 0$ and $f'(z^{k};d^k)\leq -\rho \|d^k\|^2$ if $h$ is $\rho$-strongly convex, where $f'(z^{k};d^k)$ is the classical directional derivative of $f$ at $z^{k}$ along $d^k$. Hence, $d^k$ is a `potentially' descent direction of $f$ at $z^{k}$. Particularly, if $\setC$ is a polyhedral convex set and $\setA(x^k)$ denotes the active set of $\setC$ at $x^k$, then $\setA(z^{k})\subset \setA(x^{k})$ is a necessary and sufficient condition for $d^k$ being a DC descent direction; if $\setC$ is convex but not polyhedral, then $\setA(z^{k})\subset \setA(x^{k})$ is just a necessary (not always sufficient) condition for $d^k$ being a DC descent direction. The reader is referred to \cite{Niu19higher} for more discussion on the DC descent direction and BDCA algorithm.

	BDCA for problem \eqref{prob:dcp} is summarized in \Cref{alg:BDCA}.
	\begin{algorithm}[ht!]
		\caption{BDCA}
		\label{alg:BDCA}
		\begin{algorithmic}[1]
			\STATE \textbf{Initialization:} $x^0\in \dom \partial h$;
			\FOR{$k=0,1,\ldots$}
			\STATE $y^k \in \partial h(x^k)$;
			\STATE $z^{k}\in \argmin\{g(x)-\langle x, y^k \rangle :x\in \setC\}$;
			\STATE $d^k\leftarrow z^{k}-x^{k}$;
			\STATE initialize $x^{k+1}\leftarrow z^k$;
			\IF{$\setA(z^{k})\subset \setA(x^{k})$ and $f'(z^{k};d^k) < 0$}
			\STATE $\alpha_k\leftarrow \text{LineSearch}(z^k,d^k)$;
			\STATE $x^{k+1} \leftarrow z^k + \alpha_k d^k$;
			\ENDIF
			\ENDFOR
		\end{algorithmic}
	\end{algorithm}
	\\Some comments on BDCA:
	\begin{itemize}[leftmargin=12pt]
		\item For proceeding line search, we have to check the conditions $\setA(z^k)\subset \setA(x^k)$ and $f'(z^k;d^k)<0$. In particular, if $f$ is differentiable at $z^k$, then $f'(z^k;d^k)$ is reduced to $\langle \nabla f(z^k), d^k \rangle,$ which is easy to compute.
		\item The line search procedure \verb|LineSearch|($z^k$,$d^k$) aims at finding an optimal stepsize $\alpha_k$ such that
		$$\alpha_k = \argmin \{f(z^k + \alpha d^k): z^k + \alpha d^k \in \setC, \alpha \geq 0\}.$$ 
		This problem could be solved either exactly (exact line search) or inexactly (inexact line search) depending on the problem structure and problem size. In general, finding an exact solution for $\alpha_k$ is computationally expensive and not really needed. In practice, we often find an approximate solution for $\alpha_k$ with an inexpensive procedure (e.g., Armijo-type, Goldstein-type, Wolfe-type \cite[Chapter 3]{wright2006numerical}), 
		but the exact line search will lead to the best candidate to update $x^{k+1}$. Note that for exact line search with unbounded $\setC$, the sequence $\{\alpha_k\}$ maybe unbounded. Hence, we have to study the boundedness of $\alpha_k$ which is essential to the well-definiteness of $\{x^{k}\}$ and the convergence of BDCA.    
		See successful examples in \cite{Niu19higher} for BDCA with inexact Armijo-type line search and in \cite{zhang2022boosted} for BDCA with exact line search to the higher-order moment MVSK portfolio optimization problem.  
	\end{itemize}
	BDCA enjoys the next convergence theorem:
	\begin{theorem}[Convergence theorem of BDCA, see \cite{niu2022convergence,Niu19higher}]\label{thm:convBDCA}
		Let $\{(x^k,y^k,z^k)\}$ be the sequence generated by BDCA for problem \eqref{prob:dcp} from $x^0\in \dom \partial h$. Let $g$ (resp. $h$) be convex over $\setC$ with modulus $\rho_g\geq 0$ (resp. $\rho_h\geq 0$). If either $g$ or $h$ is strongly convex over $\setC$ (i.e., $\rho_g + \rho_h>0$) and the sequence $\{(x^k,y^k,z^k)\}$ is bounded, then 
		\begin{itemize}[leftmargin=12pt]
			\item (Convergence of $\{f(x^k)\}$) the sequence $\{f(x^k)\}$ is non-increasing and convergent.
			\item (Convergence of $\{\|x^k-z^{k}\|\}$ and  $\{\|x^k-x^{k+1}\|\}$) $$\|x^k-z^{k}\|\xrightarrow{k\to \infty} 0 \quad \text{ and } \quad \|x^k-x^{k+1}\|\xrightarrow{k\to \infty} 0.$$
			\item (Subsequential convergence of $\{x^k\}$) any cluster point of the sequence $\{x^k\}$ is a DC critical point of \eqref{prob:dcp}. Moreover, if $h$ is continuously differentiable, then any cluster point of the sequence $\{x^k\}$ is a strongly DC critical point of \eqref{prob:dcp}.
			\item (Global convergence of $\{x^k\}$) furthermore, if $f$ is a KL function, $\setC$ is a semi-algebraic set, and $h$ has locally Lipschitz continuous gradient over $\setC$, then $\{x^k\}$ converges to a strongly DC critical point of \eqref{prob:dcp}, which is also a KKT point of \eqref{prob:dcp}.
		\end{itemize}
	\end{theorem}
	
	\section{DC Formulations and DCA/BDCA for \eqref{eq:seicp}}\label{sec:dcforeicp}
	In this section, we will focus on establishing DC programming formulations for \eqref{prob:seicp_log} and \eqref{prob:seicp_qcqp}, and applying DCA and BDCA for solving them.
	
	\subsection{DC formulation and DCA/BDCA for \eqref{prob:seicp_log}}
	\paragraph{\textbf{DC formulation for \eqref{prob:seicp_log}}}
	The problem \eqref{prob:seicp_log} has a DC formulation as
	$$\min \{g(x) - h(x): x\in \Omega\},$$
	where 
	\begin{equation}\label{eq:gandhforlnP}
		g(x)=\frac{\eta}{2}\|x\|_2^2-\ln(x^{\top}A x), h(x) =\frac{\eta}{2}\|x\|_2^2-\ln(x^{\top} B x), \nabla h(x)=\eta x - \frac{2B x}{x^{\top} B x}.
	\end{equation}
	We will prove that both $g$ and $h$ are \emph{strongly convex} and \emph{have Lipschitz continuous gradients} over $\Omega$ (classical definition in optimization, see e.g., \cite[Chapter 5]{beck2017first}) for some large enough $\eta$. Note that the function $\varphi_{A}(x):= -\ln(x^{\top}Ax)$ is nonconvex on $\Omega$ for $\SPD$ matrix $A$. In fact, 
	\begin{equation}
		\label{eq:gradandhessofphiA}
		\nabla \varphi_{A} (x) = - \frac{2A x}{x^{\top} A x} \quad \text{and} \quad \nabla^2 \varphi_{A} (x) = \frac{4(Ax)(Ax)^{\top}}{(x^{\top} A x)^2} - \frac{2 A}{x^{\top} A x},
	\end{equation}
	where $\nabla^2 \varphi_{A} (x)$ may not be a PD matrix over $\Omega$. A very simple and convincing example is the following one:
		Let $$A = \begin{bmatrix}
			3 & 0\\
			0 & 1
		\end{bmatrix}\in \SPD.$$
	Then taking $\bar{x}=[1,0]^{\top}\in \Omega$, we get 
	$$\nabla^2 \varphi_{A} (\bar{x}) = \begin{bmatrix}
		2 & 0\\
		0 & -\frac{2}{3}
	\end{bmatrix},$$
which is obviously not a PD matrix. Hence, $\varphi_{A}$ is nonconvex on $\Omega$. 

	The next lemma shows that there exists some large enough $\eta$ such that both $g$ and $h$ are strongly convex over $\Omega$. 
	\begin{lemma}\label{lem:strongconvexityofgandh}
		Let \begin{equation}
			\label{eq:etabar}
			\bar{\eta} = 4n \max\{\kappa_A^2,\kappa_B^2\},
		\end{equation}
		where $\kappa_A$ and $\kappa_B$ are condition numbers of $A$ and $B$. 
		Then for all $\eta \geq \bar{\eta}$, both $g$ and $h$ defined in \eqref{eq:gandhforlnP} are strongly convex over $\Omega$. 
	\end{lemma}
\begin{proof}
	We will show that $\bar{\eta}\geq \max\{\rho(\nabla^2\varphi_{A}(x)), \rho(\nabla^2\varphi_{B}(x))\}$ for all $x\in \Omega$, where $\rho(M)$ denotes the spectral radius of $M$. We first consider the matrix $A$, it follows from the positive definiteness of $A$ that $\forall x\in \Omega$,
	$$\rho(\nabla^2 \varphi_{A}(x)) < \rho\left( \nabla^2 \varphi_{A}(x) + \frac{2 A}{x^{\top} A x}\right)= \rho\left(\frac{4(Ax)(Ax)^{\top}}{(x^{\top} A x)^2}\right), $$
	where the first strict inequality is due to the renowned monotonicity theorem \cite[Corollary 4.3.12]{horn2012matrix} and the second equality comes from \eqref{eq:gradandhessofphiA}. Then,
	$$\rho\left(\frac{4(Ax)(Ax)^{\top}}{(x^{\top} A x)^2}\right) = \frac{4\rho((Ax)(Ax)^{\top})}{(x^{\top} A x)^2}= \frac{4~\|Ax\|_2^2}{(x^{\top} A x)^2}.$$
	Let $A=P^{\top}\Lambda P$ be the spectral decomposition of the $\SPD$ matrix $A$, $\Lambda = \diag(\lambda_1,\ldots,\lambda_n)$ where $\lambda_1\leq \cdots\leq \lambda_n$ are eigenvalues of $A$ with $\lambda_1>0$, and denote $y=Px$. Then 
	\begin{equation}
		\label{eq:ubofAx}
		\|Ax\|_2^2 = y^{\top} \Lambda^2 y = \sum_{i=1}^{n}\lambda_i^2 y_i^2 \leq \lambda_n^2 \|y\|_2^2 = \lambda_n^2 \|x\|_2^2,
	\end{equation}
	and
	\begin{equation}
		\label{eq:lbofxAx}
	(x^{\top}Ax)^2 = (y^{\top}\Lambda y)^2 = (\sum_{i=1}^{n}\lambda_i y_i^2)^2 \geq \lambda_1^2 (\sum_{i=1}^{n}y_i^2)^2 = \lambda_1^2 \|y\|_2^4 = \lambda_1^2 \|x\|_2^4.
	\end{equation}
	Hence,
	$$ \rho(\nabla^2 \varphi_{A}(x)) < \frac{4~\|Ax\|_2^2}{(x^{\top} A x)^2} \overset{\eqref{eq:ubofAx}\eqref{eq:lbofxAx}}{\leq} \frac{4\lambda_n^2 \|x\|_2^2}{\lambda_1^2 \|x\|_2^4} =  \frac{4\kappa_A^2}{ \|x\|_2^2} \leq  4n \kappa_A^2, \quad \forall x\in \Omega.$$ 
	Similar result can be obtained for $B$ as $\rho(\nabla^2 \varphi_{B}(x))< 4n\kappa_B^2, \forall x\in \Omega$. It follows immediately that for $\eta \geq \bar{\eta}:=4n\max\{\kappa_A^2,\kappa_B^2\}$, 
	$$\nabla^2 g(x) = \eta I - \nabla^2\varphi_{A}(x)\succ 0 \quad \text{and}\quad  \nabla^2 h(x) = \eta I - \nabla^2\varphi_{B}(x)\succ 0, \quad \forall x\in  \Omega.$$
	Hence, there exists some $\sigma>0$ such that $x\mapsto g(x)-\frac{\sigma}{2}\|x\|_2^2$ and $x\mapsto h(x)-\frac{\sigma}{2}\|x\|_2^2$ are convex over $\Omega$, i.e., $g$ and $h$ are $\sigma$-strongly convex over $\Omega$ for all $\eta\geq \bar{\eta}$.
\end{proof}

	\begin{corollary}\label{cor:smoothnessofgandh}
		For all $\eta \geq \bar{\eta}$ with $\bar{\eta}$ defined in \eqref{eq:etabar}. The convex function $g$ (resp. $h$) defined in \eqref{eq:gandhforlnP} is $L_g$-smooth (resp. $L_h$-smooth) on $\Omega$ with 
		$$L_g = \eta + 2n\kappa_A \quad \text{and} \quad L_h = \eta + 2n\kappa_B.$$
	\end{corollary}
	\begin{proof}
		It follows from \Cref{lem:strongconvexityofgandh} that for all $x\in \Omega$,
		\begin{eqnarray*}
			\rho(\nabla^2 g(x)) &=& \rho(\eta I - \nabla^2\varphi_{A}(x))\\
			&=& \rho\left(\eta I + \frac{2 A}{x^{\top} A x} - \frac{4(Ax)(Ax)^{\top}}{(x^{\top} A x)^2} \right) \\
			&\leq&  \rho\left(\eta I + \frac{2 A}{x^{\top} A x}\right)\\
			&=&\eta + \frac{2 \rho(A)}{x^{\top} A x}\\
			&\leq& \eta + \frac{2\lambda_n}{\lambda_1/n} = \eta + 2n \kappa_A = L_g,
		\end{eqnarray*}
	where the last inequality is due to $\rho(A)=\lambda_n$ and $x^{\top}A x \overset{\eqref{eq:lbofxAx}}{\geq} \lambda_1 \|x\|_2^2 \geq \lambda_1/n$ for all $x\in \Omega$. 
		Hence, $x\mapsto \frac{L_g}{2}\|x\|_2^2 - g(x)$ is convex over $\Omega$, implying that $g$ is $L_g$-smooth on $\Omega$. Similar result can be obtained for $h$ being $L_h$-smooth on $\Omega$.
	\end{proof}

	Applying DCA and BDCA to \eqref{prob:seicp_log} requires solving the convex subproblems:
	\begin{equation}
		\label{prob:LnPk}
		\boxed{\min \{ \frac{\eta}{2}\|x\|_2^2-\ln(x^{\top}A x) - \langle x, \nabla h(x^k) \rangle : x\in \Omega \}} \tag{LnPk}
	\end{equation}
	with strongly convex objective function, whose optimal solution exists and is unique. 
	
	Note that 
	since $\bar{\eta}=O(n\max\{\kappa_A^2,\kappa_B^2\})$, then for \eqref{prob:LnPk} with ill-conditioned $A$ or $B$ and with very large $n$, the parameters $\bar{\eta}$, $L_g$ and $L_h$ will be also very large. In this case, solving \eqref{prob:LnPk} will become cumbersome, since a very large Lipschitz constant $L_g$ often corresponds to a very small $1/L_g$ stepsize to slow down many solution approaches to \eqref{prob:LnPk}. Furthermore, if $\mu$ is too large, one may suffer from insatiability issue when solving an ill-conditioned problem, which may be a potential drawback of the formulation. 
	
	Next, we propose applying the renowned FISTA \cite{beck2009fista} (a proximal gradient method with Nesterov's acceleration) for solving \eqref{prob:LnPk}. 
	
	\paragraph{\textbf{FISTA for solving subproblem \eqref{prob:LnPk}}}
	FISTA is an efficient algorithm for solving the composite optimization problem 
	$$\min_{x\in  \R^n} f(x):=\phi(x) + \psi(x)$$
	under the assumptions that $\phi:\R^n\to \R$ is $L_{\phi}$-smooth and convex and $\psi$ belongs to $\Gamma_0(\R^n)$. 
	Problem \eqref{prob:LnPk} can be rewritten as 
	$$\min_{x\in\R^n}  \underbrace{\frac{\eta}{2}\|x\|_2^2-\ln(x^{\top}A x) - \langle x, \nabla h(x^k) \rangle}_{\phi(x)}  + \underbrace{\chi_{\Omega}(x)}_{\psi(x)}$$
	by introducing the indicator function $\chi_{\Omega}$ into the objective function, where $\phi$ is $L_g$-smooth and convex and $\psi$ belongs to $\Gamma_0(\R^n)$. So FISTA is applicable to \eqref{prob:LnPk} as described in \Cref{alg:FISTA_for_LnPk}.
	\begin{algorithm}[ht!]
		\caption{FISTA for \eqref{prob:LnPk}}
		\label{alg:FISTA_for_LnPk}
		\begin{algorithmic}[1]
			\REQUIRE $y^0=u^0=x^k\neq 0$; $t_0=1$;
			\FOR{$i=0,1,\ldots$}
			\STATE pick $L_i>0$;\label{codeline:3_fista_LnPk}
			\STATE $u^{i+1} \leftarrow P_{\Omega}\left( y^i - \frac{1}{L_i}\nabla \phi(y^i)\right);$\label{codeline:4_fista_LnPk}
			\STATE $t_{i+1}\leftarrow \frac{1+\sqrt{1+4t_i^2}}{2}$;
			\STATE $y^{i+1} \leftarrow u^{i+1} + \left( \frac{t_i -1}{t_{i+1}}\right)(u^{i+1}-u^i)$;
			\ENDFOR
		\end{algorithmic}
	\end{algorithm}
	\\Here are some comments on \Cref{alg:FISTA_for_LnPk}:
\begin{itemize}[leftmargin=12pt]
	\item In line \ref{codeline:4_fista_LnPk}, $P_{\Omega}(v)$ denotes the \emph{simplex projection} of vector $v$, which is derived from 
	\begin{eqnarray*}
		\prox_{\psi/L_i}\left( y^i - \frac{1}{L_i}\nabla \phi(y^i)\right) &=& \prox_{\chi_{\Omega}/L_i}(y^i - \frac{1}{L_i}\nabla \phi(y^i)) \\
		&=& \argmin_{x\in \Omega}\left\{ \left\|x - \left( y^i - \frac{1}{L_i}\nabla \phi(y^i)\right) \right\|_2^2\right\}\\
		&=& P_{\Omega}\left( y^i - \frac{1}{L_i}\nabla \phi(y^i)\right),
	\end{eqnarray*}
where $\prox$ is the classical proximal operator (see e.g. \cite{beck2017first}) and   
$$\nabla \phi(y^i) = \eta (y^i-x^k) - \frac{2A y^i}{\langle y^i, A y^i\rangle} + \frac{2B x^k}{\langle x^k, B x^k\rangle}.$$
The simplex projection $P_{\Omega}(v)$ is computed by 
$$P_{\Omega}(v) = [v - \lambda^*e]_+$$
where $\lambda^*$ is a root of the equation $e^{\top}[v-\lambda^*e]_+ = 1$ (see e.g., \cite[Corollary 6.29]{2016Fast}). There are several efficient algorithms for computing the simplex project, such as the \emph{direct projection method} in \cite{1974Validation} and the \emph{Block Pivotal Principal Pivoting Algorithm} (BPPPA) in \cite{judice1992,Niu11}. See \cite{2016Fast} for an excellent review of several efficient algorithms to simplex projection with the worst case complexity of order $O(n\log(n))$.
Here, we propose using \Cref{alg:Simplex_Proj} proposed in \cite{1974Validation} (see also \cite{2016Fast}) with $O(n\log(n))$ worst case complexity for its simplicity and efficiency.
\begin{algorithm}[ht!]
	\caption{Simplex Projection}
	\label{alg:Simplex_Proj}
	\begin{algorithmic}[1]
		\REQUIRE $v\in \R^n$;
		\ENSURE $P_{\Omega}(v)$;
		\STATE sort $v$ into $z$ with $z_1\geq z_2\geq \cdots\geq z_n$;
		\STATE $N \leftarrow \max_{1\leq k\leq n} \{ k : (\sum_{r=1}^k z_r - 1)/k < z_k \}$;
		\STATE $\lambda^* \leftarrow (\sum_{r=1}^N z_r - 1)/N$;
		\RETURN $P_{\Omega}(v) = [v-\lambda^* e]_+$.	
	\end{algorithmic}
\end{algorithm} 
\item We consider two options for the choice of $L_i$ in line \ref{codeline:3_fista_LnPk}: constant and backtracking. \\
\textbf{Constant:} Fix $L_i = L_g$ for all $i$. This choice is suitable when $L_g$ is not too large.\\
\textbf{Backtracking:} 
Given two parameters $(s,r)$ with $s>0$ (an initial guess for $L_i$, expected to be smaller than $L_g$) and $r>1$ (the expansion parameter). One can start by initializing $L_{-1}=s$. Then at iteration $i$ ($i\geq 0$), by denoting the operator $T_{L_i}(y) := P_{\Omega}\left( y - \frac{1}{L_i}\nabla \phi(y)\right)$, we first set $L_i=L_{i-1}$ and test whether the inequality below is verified
\begin{equation}
	\label{eq:ineqsmoothness}
	\phi(T_{L_i}(y^i)) \leq \phi(y^i) + \langle \nabla \phi(y^i), T_{L_i}(y^i) - y^i\rangle + \frac{L_i}{2}\|T_{L_i}(y^i) - y^i\|_2^2.
\end{equation}
If yes, then we obtain a suitable $L_i$; Otherwise, we enlarge $L_i$ by $r L_i$ and test again the inequality \eqref{eq:ineqsmoothness}. This backtracking procedure is repeated until \eqref{eq:ineqsmoothness} is verified. Note that the inequality \eqref{eq:ineqsmoothness} is always satisfied for large enough $L_i$, because $\phi$ is $L_g$-smooth and this inequality holds whenever $L_i\geq L_g$. This procedure allows us to find some suitable $L_i$ smaller than $L_g$ (even without knowing $L_g$ in prior), such that the gradient step $y^i - \frac{1}{L_i}\nabla \phi(y^i)$ has some stepsize $1/L_i$ larger than the fixed stepsize $1/L_g$, which will potentially yield a better descent. 
\item It is known that FISTA has an $O(1/i^2)$ rate of convergence in function values using either constant or backtracking stepsize. The reader is refereed to \cite[Chapter 10.7]{beck2017first} for more discussion on FISTA. 
\end{itemize}
	
	\paragraph{\textbf{Exact line search in BDCA for \eqref{prob:seicp_log}}} We can compute exact line search efficiently as follows: consider the line search problem $$\alpha_k = \argmin \{f(z^k + \alpha d^k): z^k + \alpha d^k\in \Omega, \alpha\geq 0\}$$
    for $d^k\neq 0$. Then $$[z^k + \alpha d^k\in \Omega, \alpha\geq 0] \Leftrightarrow [e^{\top}(z^k + \alpha d^k) = 1, z^k + \alpha d^k \geq 0, \alpha \geq 0].$$
	It follows that 
	$$e^{\top}(z^k + \alpha d^k) = e^{\top}z^k + \alpha e^{\top} d^k = 1 $$
		since $z^k, x^k\in \Omega, \forall k\geq 1$ implies that $e^{\top}z^k=1$, $e^{\top}x^k=1$ and $e^{\top}d^k = e^{\top}(z^k-x^k) = 0.$
	$$[z^k + \alpha d^k\geq 0, \alpha\geq 0] \Leftrightarrow 0\leq \alpha \leq -\frac{z^k_i}{d^k_i}, \quad \forall i\in \mathcal{I}^k,$$
	where $\mathcal{I}^k:=\{i\in \{1,\ldots,n\}: d^k_i<0\}$.
	Hence, we obtain a bound for $\alpha$ as:
	$$0\leq \alpha \leq \bar{\alpha}_k$$
	with
	\begin{equation}
		\label{eq:alphabark_LnP}
		\boxed{\bar{\alpha}_k: = \min\left\{ -\frac{z^k_i}{d^k_i}, i\in \mathcal{I}^k \right\}}
	\end{equation}
	under the convention that $\min \emptyset = \infty$,
	and the line search is simplified as 
	\begin{equation}\label{prob:linesearch_LnP}
		\alpha_k = \argmin \{f(z^k + \alpha d^k): 0 \leq \alpha\leq  \bar{\alpha}_k\}.
	\end{equation} 
	\begin{proposition}\label{prop:linesearch_LnP}
		The exact line search in BDCA for \eqref{prob:seicp_log} at $z^k$ along $d^k$ is computed by
		$$\alpha_k = \argmin_{\alpha} \{q(\alpha): x\in \{0,\bar{\alpha}_k\}\cup \mathcal{Z}\},$$ 
		where $$\begin{cases}
			q(x)=(a_1x^2+b_1x+c_1)/(a_2x^2+b_2x+c_2),\\
			a_1=\langle d^k, B d^k\rangle, b_1=2\langle z^k, B d^k \rangle, c_1=\langle z^k, B z^k\rangle,\\
			a_2=\langle d^k, A d^k\rangle, b_2=2\langle z^k, A d^k \rangle, c_2=\langle z^k, A z^k\rangle,\\
			\bar{\alpha}_k = \min\left\{ -z^k_i/d^k_i, i\in \mathcal{I}^k \right\} \text{ with } \mathcal{I}^k=\{i\in \{1,\ldots,n\}: d^k_i<0\},\\
		\end{cases}$$ and $\mathcal{Z}$ is the set of all real roots of the binomial $$(a_1b_2-a_2b_1)x^2 + 2(a_1c_2-a_2c_1)x + b_1c_2-b_2c_1$$ within the interval $\left[0,\bar{\alpha}_k\right]$. Then we set $x^{k+1} = z^k+ \alpha_k d^k$.
	\end{proposition}
	\begin{proof}
		Consider the line search problem \eqref{prob:linesearch_LnP} whose objective function is $$f(z^k+\alpha d^k)=\ln\left( \frac{(z^k + \alpha d^k)^{\top}B (z^k + \alpha d^k)}{(z^k + \alpha d^k)^{\top}A (z^k + \alpha d^k)} \right).$$
		Then, by the strictly increasing of the function $\ln$, we get that 
		\begin{eqnarray*}
			\alpha_k &=& \argmin_{\alpha} \{f(z^k + \alpha d^k): 0 \leq \alpha\leq  \bar{\alpha}_k\}\\
			&=&\argmin_{\alpha} \left\{\frac{(z^k + \alpha d^k)^{\top}B (z^k + \alpha d^k)}{(z^k + \alpha d^k)^{\top}A (z^k + \alpha d^k)}: 0\leq \alpha \leq \bar{\alpha}_k\right\}\\
			&=&\argmin_{\alpha} \left\{\frac{\langle d^k, B d^k\rangle\alpha^2 +2\langle z^k, B d^k \rangle\alpha + \langle z^k, B z^k\rangle}{\langle d^k, A d^k\rangle\alpha^2 +2\langle z^k, A d^k \rangle \alpha + \langle z^k, A z^k\rangle}: 0\leq \alpha \leq \bar{\alpha}_k\right\}.
		\end{eqnarray*}
		Now, consider the optimization problem in form of 
		$$x^*= \argmin \Biggl\{ q(x):=\frac{a_1x^2+b_1x+c_1}{a_2x^2+b_2x+c_2}:\quad  0\leq x\leq \bar{x}
		\Biggr\},$$
		where $a_1x^2+b_1x+c_1$ and $a_2x^2+b_2x+c_2$ are strictly positive for all $x\in [0,\bar{x}]$. The derivative of $q$ is given by
		$$q'(x) = \frac{(a_1b_2-a_2b_1)x^2 + 2(a_1c_2-a_2c_1)x + b_1c_2-b_2c_1}{(a_2x^2+b_2x + c_2)^2},$$
		whose roots are exactly roots of the binomial 
		\begin{equation}
			\label{eq:polyofnumerator}
			(a_1b_2-a_2b_1)x^2 + 2(a_1c_2-a_2c_1)x + b_1c_2-b_2c_1,
		\end{equation}
		which can be computed without any difficulty. Let $\mathcal{Z}$ be the set of all real roots of this binomial within the interval $[0,\bar{x}]$. Clearly, all minima of $q$ over $[0,\bar{x}]$ should be included in $\{0,\bar{x}\}\cup \mathcal{Z}$. Then we get 
		$$x^*= \argmin_{x} \{q(x): x\in \{0,\bar{x}\}\cup \mathcal{Z}\}.$$
		Applying this to compute $\alpha_k$, we get
		$$\alpha_k = \argmin_{\alpha} \{q(\alpha): \alpha\in \{0,\bar{\alpha}_k\}\cup \mathcal{Z}\},$$ 
		where $$\begin{cases}
			a_1=\langle d^k, B d^k\rangle, b_1=2\langle z^k, B d^k \rangle, c_1=\langle z^k, B z^k\rangle,\\
			a_2=\langle d^k, A d^k\rangle, b_2=2\langle z^k, A d^k \rangle, c_2=\langle z^k, A z^k\rangle,
		\end{cases}$$ and $\mathcal{Z}$ is the set of all real roots of \eqref{eq:polyofnumerator} within $[0,\bar{\alpha}_k]$ where $\bar{\alpha}_k$ is given by \eqref{eq:alphabark_LnP}. 
	\end{proof}
	
	\paragraph{\textbf{DCA/BDCA for \eqref{prob:seicp_log}}} Now, we describe the BDCA for \eqref{prob:seicp_log} in \Cref{alg:BDCA_for_LnP}.
	\begin{algorithm}[ht!]
		\caption{BDCA for \eqref{prob:seicp_log}}
		\label{alg:BDCA_for_LnP}
		\begin{algorithmic}[1]
			\REQUIRE $x^0\neq 0$; $\eta\geq \bar{\eta}$ with $\bar{\eta}$ defined in \eqref{eq:etabar};
			\FOR{$k=0,1,\ldots$}
			\STATE compute $z^{k}\leftarrow \argmin \{ \frac{\eta}{2}\|x\|_2^2-\ln(x^{\top}A x) - \langle x, \nabla h(x^k) \rangle : x\in \Omega \}$ via FISTA;\label{codeline:2_bdca_log} 
			\STATE $d^k\leftarrow z^{k}-x^{k}$;
			\STATE initialize $x^{k+1}\leftarrow z^k$; \label{codeline:4_bdca_log}
			\IF{$\setA(z^{k})\subset \setA(x^{k})$ and $\left\langle \frac{B z^k}{\langle z^k, B z^k \rangle} - \frac{A z^k}{\langle z^k, A z^k \rangle}, d^k \right\rangle < 0$}\label{codeline:5_bdca_log}
			\STATE $\alpha_k \leftarrow \argmin_{\alpha} \{q(\alpha): x\in \{0,\bar{\alpha}_k\}\cup \mathcal{Z}\}$ as described in \Cref{prop:linesearch_LnP};\label{codeline:6_bdca_log}
			\STATE $x^{k+1} \leftarrow z^k + \alpha_k d^k;$
			\ENDIF\label{codeline:8_bdca_log}
			\ENDFOR
		\end{algorithmic}
	\end{algorithm}
	Some comments on \Cref{alg:BDCA_for_LnP} are described as follows: 
	\begin{itemize}[leftmargin=12pt]
		\item DCA for \eqref{prob:seicp_log} is just BDCA without the codes from line \ref{codeline:5_bdca_log} to \ref{codeline:8_bdca_log}.
		\item In line \ref{codeline:2_bdca_log}, $z^k$ is computed by solving the convex subproblem \eqref{prob:LnPk} via FISTA. In fact, only an approximate solution for $z^k$ (i.e., a feasible solution better than $x^k$ with a smaller objective value) is required to compute.
		\item In line \ref{codeline:5_bdca_log}, the active set $\setA(x^k)$ is defined by $\{i=1,\ldots,n: x^k_i=0\}$. The second condition is derived from $f'(z^k;d^k)<0$ with 
		$$f'(z^k;d^k) = \langle \nabla f(z^k), d^k \rangle = \left\langle \frac{2B z^k}{\langle z^k, B z^k \rangle} - \frac{2A z^k}{\langle z^k, A z^k \rangle}, d^k \right\rangle.$$ These are necessary and sufficient conditions for $d^k$ being a DC descent direction for polyhedral convex set. Note that, without checking these conditions and performing the line search all the time (i.e., removing the lines \ref{codeline:5_bdca_log} and \ref{codeline:8_bdca_log}), this algorithm still works fine, but we strongly suggest checking these conditions which often leads to better numerical performance in practice.  
		\item \Cref{thm:convDCA} and \Cref{thm:convBDCA} for the convergence of DCA and BDCA are fulfilled since $f$ is a KL function, both $g$ and $h$ are strongly convex on $\Omega$ due to \Cref{lem:strongconvexityofgandh}, and $h$ has locally Lipschitz continuous gradient on $\Omega$ due to \Cref{cor:smoothnessofgandh}. 
	\end{itemize}
	
	Note that we can solve \eqref{prob:eicp_frac} by the same algorithms (namely, DCA and BDCA) described in this subsection through the logarithmic formulation \eqref{prob:seicp_log} based on the equivalence between \eqref{prob:eicp_frac} and \eqref{prob:seicp_log}. 
	
	\subsection{DC formulation and DCA/BDCA for \eqref{prob:seicp_qcqp}}
	\paragraph{\textbf{DC formulation for \eqref{prob:seicp_qcqp}}}
	The problem \eqref{prob:seicp_qcqp} is a convex maximization problem with a trivial DC formulation in minimization form as 
	$$\min \{f(x)=g(x) - h(x): x^{\top}   B   x\leq 1, x\geq 0 \},$$
	where $$g(x)= 0 , \quad h(x) =x^{\top}   A   x \quad \text{ and }\quad \nabla h (x) = 2 A x.$$ 
	
	Applying DCA and BDCA to this DC decomposition requires solving the linear minimization subproblems over a compact convex set (the intersection of an ellipsoid and the nonnegative orthant) as
	\begin{equation}
		\label{prob:QPk}
		\boxed{\min \{ \langle -2Ax^k, x \rangle : x^{\top}   B   x\leq 1, x\geq 0 \},}\tag{QPk}
	\end{equation}
	which can be efficiently solved by many quadratic or second order cone programming solvers such as GUROBI, CPLEX and MOSEK.  
	
	\paragraph{\textbf{Exact line search in BDCA for \eqref{prob:seicp_qcqp}}} We follow a similar way as in \eqref{prob:seicp_log} to compute the exact line search. Consider the line search problem
	$$\alpha_k = \argmin \{f(z^k + \alpha d^k): (z^k + \alpha d^k)^{\top}B(z^k + \alpha d^k)\leq 1, z^k + \alpha d^k\geq 0, \alpha \geq 0\}.$$
	Then for all $d^k\neq 0$, we have that the binomial $(z^k + \alpha d^k)^{\top}B(z^k + \alpha d^k)\leq 1$ is equivalent to
	$$\alpha\leq \frac{-\langle z^k, B d^k\rangle + \sqrt{\langle z^k, B d^k\rangle^2 - \langle d^k, B d^k\rangle (\langle z^k, B z^k\rangle-1)}}{\langle d^k, B d^k\rangle}$$
	and
	$$[z^k + \alpha d^k\geq 0, \alpha\geq 0] \Leftrightarrow 0\leq \alpha \leq -\frac{z^k_i}{d^k_i}, \quad \forall i\in \mathcal{I}^k,$$
	where $\mathcal{I}^k:=\{i\in \{1,\ldots,n\}: d^k_i<0\}$.
	Combining them, we obtain a bound for $\alpha$ as:
	$$0\leq \alpha \leq \bar{\alpha}_k$$
	with
	\begin{equation}
		\label{eq:alphabark_qp}
		\boxed{\bar{\alpha}_k: = \min\left\{ \frac{-\langle z^k, B d^k\rangle + \sqrt{\langle z^k, B d^k\rangle^2 - \langle d^k, B d^k\rangle (\langle z^k, B z^k\rangle-1)}}{\langle d^k, B d^k\rangle}, -\frac{z^k_i}{d^k_i}, i\in \mathcal{I}^k \right\}.}
	\end{equation}
	\begin{proposition}\label{prop:linesearch_qcqp}
		The exact line search in BDCA for \eqref{prob:seicp_qcqp} at $z^k$ along $d^k$ is computed by
		$$\alpha_k = \begin{cases}\bar{\alpha}_k,& \text{if } f(z^k+\bar{\alpha}_kd^k)< f(z^k),\\
			0,& \text{otherwise},
		\end{cases}$$
		with $\bar{\alpha}_k$ computed in \eqref{eq:alphabark_qp}. Then we set $x^{k+1} = z^k+ \alpha_k d^k$.
		
	\end{proposition}
	\begin{proof}
		The result follows immediately from the concavity of $\alpha\mapsto f(z^k+\alpha d^k) = - (z^k+\alpha d^k)^{\top} A (z^k+\alpha d^k)$ over the interval $0\leq \alpha\leq \bar{\alpha}_k$.
	\end{proof}
	\paragraph{\textbf{DCA/BDCA for \eqref{prob:seicp_qcqp}}} Now, we describe BDCA for \eqref{prob:seicp_qcqp} in \Cref{alg:BDCA_for_QP}.
	\begin{algorithm}[ht!]
		\caption{BDCA for \eqref{prob:seicp_qcqp}}
		\label{alg:BDCA_for_QP}
		\begin{algorithmic}[1]
			\REQUIRE $x^0\neq 0$;
			\FOR{$k=0,1,\ldots$}
			\STATE $z^{k}\in \argmin\{-\langle x, 2Ax^k \rangle :x^{\top}   B   x\leq 1, x\geq 0 \}$;\label{codeline:2_bdca_qp}
			\STATE initialize $x^{k+1}\leftarrow z^k$; \label{codeline:3_bdca_qp}
			\STATE $d^k\leftarrow z^{k}-x^{k}$; \label{codeline:4_bdca_qp}
			\IF{$\setA(z^{k})\subset \setA(x^{k})$ and $-\langle 2 A z^k, d^k \rangle < 0$} \label{codeline:5_bdca_qp}
			\STATE $\mathcal{I}^k\leftarrow \{i\in \{1,\ldots,n\}: d^k_i<0\};$
			\STATE $\bar{\alpha}_k \leftarrow \min\left\{ \frac{-\langle z^k, B d^k\rangle + \sqrt{\langle z^k, B d^k\rangle^2 - \langle d^k, B d^k\rangle (\langle z^k, B z^k\rangle-1)}}{\langle d^k, B d^k\rangle}, -\frac{z^k_i}{d^k_i}, i\in \mathcal{I}^k \right\};$
			\IF{$f(z^k + \bar{\alpha}_k d^k)<f(z^k)$}
			\STATE $x^{k+1} \leftarrow z^k + \bar{\alpha}_k d^k;$
			\ENDIF
			\ENDIF\label{codeline:11_bdca_qp}
			\ENDFOR
		\end{algorithmic}
	\end{algorithm}
	\\Some comments on \Cref{alg:BDCA_for_QP} are summarized below: 
	\begin{itemize}[leftmargin=12pt]
		\item DCA for \eqref{prob:seicp_qcqp} is just BDCA without codes from line \ref{codeline:4_bdca_qp} to \ref{codeline:11_bdca_qp}.
		\item In line \ref{codeline:5_bdca_qp}, $\setA(z^k)\subset \setA(x^k)$ and $-\langle 2 A z^k, d^k \rangle < 0$ (since $f'(z^k;d^k) = \langle \nabla f(z^k),d^k \rangle = -\langle 2 A z^k, d^k \rangle$) serve as necessary conditions for $d^k$ being a DC descent direction. If one of the condition is not satisfied, then the line search is not needed. 
		\item The initial point $x^0$ can be taken arbitrarily as any nonzero point in $\R^n$. The nonsingularity of $A$ ensures that the coefficient of the linear objective function $-2Ax^k\neq 0$ whenever $x^k\neq 0$. In the case where $x^k=0$ for some $k$, then $z^k$ could be any feasible point of the convex subproblem in line \ref{codeline:3_bdca_qp}, and we suggest taking  $$z^k = \frac{\zeta}{\sqrt{\langle \zeta,B\zeta\rangle }}$$ with a random nonnegative and nonzero vector $\zeta$ in $\R^n$. This suggestion is also applicable to the classical DCA.
		\item \Cref{thm:convDCA} and \Cref{thm:convBDCA} for the convergence of DCA and BDCA are fulfilled since $f$ (as a quadratic function) is indeed a KL function, the constraint $\{x\in \R^n: x^{\top}Bx\leq 1, x\geq 0\}$ is a semi-algebraic set, $h$ is strongly convex and has globally Lipschitz continuous gradient over $\R^n$. 
	\end{itemize}
	
	\section{Extension to SQEiCP}\label{sec:SQEiCP}
	 Consider the extension \eqref{eq:sqeicp}. Let us denote $\SQEiCP{A,B,C}$ for the solution set of \eqref{eq:sqeicp}. The next hypothesis is a sufficient condition for the feasibility of $\SQEiCP{A,B,C}$ \cite{Bras16}:
	
	\begin{hypothesis}\label{hyp2}
		$A\in \SPD$, $B$ and $C$ are real symmetric matrices with $C\notin S_0=\{C\in \R^{n\times n}: \exists x\neq 0, x \geq 0, C  x \geq 0\}$.
	\end{hypothesis}
	\begin{theorem}[See \cite{Bras16}]\label{thm:existenceofsolutionsforqeicp}
		Under \Cref{hyp2}, $\SQEiCP{A,B,C}$ admits at least one positive and one negative quadratic complementary eigenvalues, and $0$ is not a quadratic complementary eigenvalue.
	\end{theorem}

	Checking whether $C\in S_0$ is easy, which reduces to solving the feasibility problem of the linear program defined by:
	$$\min\{e^{\top} x: C  x\geq 0, e^{\top} x = 1, x\geq 0\}.$$
	In particular, we often suppose that $-C\in \SPD$, which implies $C\notin S_0$.

	\subsection{From SQEiCP to SEiCP}
	
	We can prove that SQEiCP under \Cref{hyp2} is equivalent to two SEiCPs with complementary eigenvector $(x,y)\in \R^n\times \R^n$ and complementary eigenvalue $\lambda> 0$ verifying
	
	\begin{equation}\label{eq:2nEiCP-1}
		 \begin{cases}
			\lambda D   \begin{bmatrix}
				y\\ x 
			\end{bmatrix} - G   \begin{bmatrix}
				y\\ x 
			\end{bmatrix} = \begin{bmatrix}
				w\\ v 
			\end{bmatrix}, \\
			e^{\top}   \begin{bmatrix}
				y\\ x 
			\end{bmatrix} = 1, \\
			y^{\top}   w + x^{\top}   v = 0,\\
			(x,y,v,w,\lambda)\geq 0,
		\end{cases}\tag*{\SEiCP{G,D}}
	\end{equation}
	or
	\begin{equation}\label{eq:2nEiCP-2}
		\begin{cases}
			\lambda D   \begin{bmatrix}
				y\\ x 
			\end{bmatrix} - H   \begin{bmatrix}
				y\\ x 
			\end{bmatrix} = \begin{bmatrix}
				w\\ v 
			\end{bmatrix}, \\
			e^{\top}   \begin{bmatrix}
				y\\ x 
			\end{bmatrix} = 1, \\
			y^{\top}   w + x^{\top}   v = 0, \\
			(x,y,v,w,\lambda)\geq 0,
		\end{cases}\tag*{\SEiCP{H,D}}
	\end{equation}
	where $D$, $G$, $H$ are matrices of augmented size $\R^{2n\times 2n}$ defined in the next \Cref{thm:SQEiCP=SEiCP}, which is similar to \cite[Proposition 1]{Bras16} for asymmetric QEiCP. 
	
	\begin{theorem}\label{thm:SQEiCP=SEiCP}
		Let $A\in \SPD$ and $-C\in \SPD$. Then $\SQEiCP{A,B,C}$ is equivalent to the two SEiCP formulations $\SEiCP{G,D}$ and $\SEiCP{H,D}$ with 
		\[D = \begin{bmatrix}
			A & 0\\
			0 & -C
		\end{bmatrix}, 
		G = \begin{bmatrix}
			-B & -C\\
			-C & 0
		\end{bmatrix},
		H = \begin{bmatrix}
			B & -C\\
			-C & 0
		\end{bmatrix},\]
		in the sense that:\\
		$(i)$ For all $((y,x),\lambda)\in \SEiCP{G,D}$ (resp. $\SEiCP{H,D}$), we have
		\begin{enumerate}
			\item[$(1a)$] $v=0, y=\lambda x$ and $\lambda>0$.   
			\item[$(1b)$] $((1+\lambda)x,\lambda)\in \SQEiCP{A,B,C}$ (resp. $((1+\lambda)x,-\lambda)\in \SQEiCP{A,B,C}$).
		\end{enumerate}
		$(ii)$ Conversely, for all $(x,\lambda)\in \SQEiCP{A,B,C}$, then $\lambda\neq 0$ and
		\begin{enumerate}
			\item[$(2a)$] If $\lambda>0$, then $(z,\lambda)\in \SEiCP{G,D}$ with $z=(1+\lambda)^{-1}(\lambda x,x)$.
			\item[$(2b)$] If $\lambda<0$, then $(z,-\lambda)\in \SEiCP{H,D}$ with $z=(1-\lambda)^{-1}(-\lambda x,x)$.
		\end{enumerate}
	\end{theorem} 
	\begin{proof} We will prove $(1a)$ and $(1b)$ for $\SEiCP{G,D}$ (the results for $\SEiCP{H,D}$ can be proved in a similar way). Let $(x,y,v,w,\lambda)$ be a solution of $\SEiCP{G,D}$. \\
		\underline{Prove $(1a)$:} We get from
		\[\lambda D   \begin{bmatrix}
			y\\ x 
		\end{bmatrix} - G   \begin{bmatrix}
			y\\ x 
		\end{bmatrix} = \begin{bmatrix}
			w\\ v 
		\end{bmatrix}
		\]
		that 
		\begin{equation}\label{eq:auxeq1}
			\lambda A   y + B  y + C  x = w
		\end{equation}
		and 
		\begin{equation}\label{eq:auxeq1bis}
			C  (y - \lambda x) = v.
		\end{equation}
		Under the hypothesis $-C\in\SPD$, we have $C$ is invertible, and $-C^{-1}\in \SPD$. Then we multiply $C^{-1}$ in \eqref{eq:auxeq1bis} to get
		\begin{equation}\label{eq:auxeq2}
			y - \lambda x = C^{-1}  v.
		\end{equation}
		$\rhd$ $v=0$ can be proved by contradiction as follows: 
		supposing that $v\neq 0$, then we multiply $v^{\top}$ in \eqref{eq:auxeq2} to get 
		\begin{equation}\label{eq:auxeq3}
			v^{\top}   (y - \lambda x) = v^{\top}   C^{-1}   v.
		\end{equation}
		On the left part of \eqref{eq:auxeq3}, since $y^{\top}   w + x^{\top}   v = 0$, we replace $x^{\top}   v$ by $-y^{\top}   w$ to write $v^{\top}   (y - \lambda x)$ as
		$y^{\top} (v + \lambda w)$, which is nonnegative since $(y,v,w,\lambda)\geq 0$. \\
		On the right part of \eqref{eq:auxeq3}, we have $v^{\top}   C^{-1}   v<0$ for all $v\neq 0$ since $-C\in \SPD$. \\
		Therefore, we get  
		$$0\leq v^{\top}   (y - \lambda x) = v^{\top}   C^{-1}   v < 0.$$
		Contradiction! Hence $v=0$.\\
		$\rhd$ It follows from $\eqref{eq:auxeq2}$ and $v=0$ that $y=\lambda x$. \\
		$\rhd$ For proving $\lambda >0$, it is sufficient to prove $\lambda\neq 0$ since $\lambda\geq 0$. By contradiction, supposing $\lambda = 0$, then $y=\lambda x=0$, and \eqref{eq:auxeq1} is reduced to $C  x = w\geq 0$. Since $0\neq x\geq 0$, we get $x^{\top}   C   x \geq 0$. This is impossible since $-C\in \SPD$. Hence $\lambda>0$.\\
		\underline{Prove $(1b)$:} Let $\bar{x}=(1+\lambda)x$, we can verify that $(\bar{x},\lambda)\in \SQEiCP{A,B,C}$:
		\begin{itemize}[leftmargin=12pt]
			\item[$\bullet$] $\bar{x}=(1+\lambda) x \geq 0$ since $\lambda>0$ and $x\geq 0$.
			\item[$\bullet$] $\bar{w}: = \lambda^2 A  \bar{x} + \lambda B  \bar{x}+C  \bar{x} = (1+\lambda) (\lambda^2 A   x +\lambda B   x + C   x) \overset{(1a)}{=} (1+\lambda) (\lambda A   y +B   y + C   x)\overset{\eqref{eq:auxeq1}}{=}(1+\lambda) w$. Then $\bar{w}\geq 0$ follows from $w\geq 0$, $\lambda>0$ and $\bar{w}=(1+\lambda)w$.
			\item[$\bullet$] Due to $y=\lambda x$, $v=0$ and $\lambda >0$, the expression $y^{\top}   w + x^{\top}   v=0$ is reduced to $x^{\top}   w=0$. Then $\bar{x}^{\top}   \bar{w} = (1+\lambda)^2 x^{\top}  w = 0$.
			\item[$\bullet$] $e^{\top}   \bar{x} = (1+\lambda) e^{\top}   x = e^{\top}   x + \lambda e^{\top}   x = e^{\top}   x + e^{\top}   y = 1$.		
		\end{itemize}
		\underline{Conversely}: We first prove $\lambda\neq 0$ for all $(x,\lambda)\in \SQEiCP{A,B,C}$ by contradiction. Supposing that $\lambda=0$, then $\SQEiCP{A,B,C}$ is reduced to $$w=C  x, x^{\top}   w = 0, e^{\top}   x = 1, (x,w)\geq 0,$$ implying that $x^{\top} w =x^{\top}   C  x = 0$ with $x\neq 0$. Clearly, this is impossible for $-C\in \SPD$. Hence, $\lambda\neq 0$. \\
		\underbar{Prove (2a) and (2b):}  If $(x,\lambda)\in \SQEiCP{A,B,C}$ with $\lambda >0$, then by taking $z=(1+\lambda)^{-1}(\lambda x, x)$, we can check that $(z,\lambda)\in \SEiCP{G,D}$ since 
		$$\begin{aligned}
			\lambda D   z - G   z &= \lambda (1+\lambda)^{-1} \begin{bmatrix}
				A & 0\\
				0 & -C
			\end{bmatrix} \begin{bmatrix}
				\lambda x\\
				x
			\end{bmatrix} - (1+\lambda)^{-1}
			\begin{bmatrix}
				-B & -C\\
				-C & 0
			\end{bmatrix} \begin{bmatrix}
				\lambda x\\
				x
			\end{bmatrix}\\			
			& = \begin{bmatrix}
				(1+\lambda)^{-1} (\lambda^2 Ax + \lambda B x +C x)\\
				0
			\end{bmatrix} =
			\begin{bmatrix}
				w\\ v 
			\end{bmatrix},
		\end{aligned}$$
		$$e^{\top}z = (1+\lambda)^{-1}(\lambda e^{\top} x + e^{\top} x) = (1+\lambda)^{-1} (1+\lambda) = 1,$$
		$$z^{\top} \begin{bmatrix}
			w\\ v
		\end{bmatrix} = (1+\lambda)^{-1} (\lambda x^{\top} w + x^{\top}v) = (1+\lambda)^{-2} \lambda x^{\top} (\lambda^2 Ax + \lambda B x +C x) = 0,$$
		$$z = (1+\lambda)^{-1}(\lambda x, x)\geq 0, w = (1+\lambda)^{-1} (\lambda^2 Ax + \lambda B x +C x) \geq 0, v = 0, \lambda >0.$$
		$(2b)$ can be verified in a similar way.
	\end{proof}
	
	The next corollary is an immediate consequence of \Cref{thm:SQEiCP=SEiCP}.
	\begin{corollary}\label{cor:lambdaforSQEiCPandSEiCP}
		Let $A\in \SPD$ and $-C\in \SPD$. Then,
		\begin{enumerate}
			\item[(i)] any $\lambda$-component of $\SQEiCP{A,B,C}$ is either a $\lambda$-component of $\SEiCP{G,D}$ or a $\lambda$-component of $\SEiCP{H,D}$.
			\item[(ii)] for any $((y,x),\lambda)$ solution of $\SEiCP{G,D}$, $\lambda$ is a positive $\lambda$-component of $\SQEiCP{A,B,C}$.
			\item[(ii)] for any $((y,x),\lambda)$ solution of $\SEiCP{H,D}$, $-\lambda$ is a negative $\lambda$-component of $\SQEiCP{A,B,C}$.		 
		\end{enumerate} 
	\end{corollary} 
	These results ensure that we can solve $\SEiCP{G,D}$ (resp. $\SEiCP{H,D}$) to find positive (resp. negative) quadratic complementary eigenvalues for $\SQEiCP{A,B,C}$. 
	
	\subsection{DCA/BDCA for \eqref{eq:sqeicp}}
	The SEiCP formulations for \eqref{eq:sqeicp} given in \Cref{thm:SQEiCP=SEiCP} have an SPD matrix $D$, but $G$ and $H$ may not be SPD. Thanks to \Cref{thm:1}, we can convert them to equivalent SEiCPs verifying  \Cref{hyp1} as :
	$$((y,x),\lambda)\in \SEiCP{G,D} \Leftrightarrow ((y,x),\lambda+\mu_{GD})\in \SEiCP{G+\mu_{GD} D, D}$$ with
	$\mu_{GD}  > \min \{\mu: G+\mu D \succeq 0 \}$, 
	and
	$$((y,x),\lambda)\in \SEiCP{H,D} \Leftrightarrow ((y,x),\lambda+\mu_{HD})\in \SEiCP{H+\mu_{HD} D, D},$$
	with 
	$\mu_{HD}  > \min \{\mu: H+\mu D \succeq 0 \}$, 
	where $G+\mu_{GD} D$ and $H+\mu_{HD} D$ are SPD matrices. Then, we can apply DCA and BDCA presented in \Cref{sec:dcforeicp} to solve $\SEiCP{G+\mu_{GD} D, D}$ (resp.  $\SEiCP{H+\mu_{HD} D, D}$) for quadratic complementary eigenvalues. 	
	
	\section{Numerical Simulations}\label{sec:Simulations}
	In this section, we will report some numerical results of DCA and BDCA for solving \eqref{eq:seicp} and \eqref{eq:sqeicp}. Our codes are implemented on MATLAB 2021a and tested on a laptop equipped with 64 bits Windows
	10, i7-10870H 2.20GHz CPU and 32 GB of RAM. The codes are available at \url{https://github.com/niuyishuai/BDCA_SEICP_SQEICP}. We compare our methods with  
	KNITRO v11.1.0 \cite{knitro}, FILTERSD v1.0 \cite{filtersd} and MATLAB FMINCON on both \eqref{prob:seicp_log} and \eqref{prob:seicp_qcqp} formulations. Note that a global optimization solver such as BARON is not necessary since only a stationary point is needed.
		
	\textbf{SEiCP datasets:} Two sets of test problems are considered, where $B$ is taken as the identity matrix. 
	\begin{itemize}[leftmargin=12pt]
		\item In the first test set, the matrix $A$ is randomly generated with elements uniformly distributed in the intervals $[-1,1]$ and $[-10,10]$. These problems are denoted by \texttt{RANDEICP(k,m,n)}, where $k$ and $m$ are the end-points of the chosen interval for matrix generation, and $n$ is the order of the matrices taken from medium to large size in $\{50,100,200,400,600,800\}$. The condition number of $A$ is of order $O(10)$ and $O(10^2)$ for $(k,m)=[-1,1]$ and $[-10,10]$ respectively. 
		\item In the second test set, 
		the matrix $A$ is taken from the \emph{Matrix Market} repository NEP (Non-Hermitian Eigenvalue Problem) collection, in which we choose $13$ matrices with order $n$ from $100$ to $800$, where $n$ is indicated in the problem name, e.g., $n=800$ for \verb|NEP-rdb800l|. These matrices come from various fields of real applications (see \url{https://math.nist.gov/MatrixMarket} for more information). For asymmetric NEP matrix, we generate symmetric $A$ by taking $(A+A^{\top})/2$. 
	\end{itemize}
Note that we convert $A$ in $\SEiCP{A,B}$ to be $\SPD$ by setting $A = A+\mu_{AB} B$ where 
$$\mu_{AB} =\min \{\mu : A + \mu B \succeq 0\}+1$$
is a semidefinite program and solved by MOSEK 9.2. 
	
	\textbf{SQEiCP datasets:} we consider a set of randomly generated test problems where $A$ is taken as the identity matrix, $B$ is a sparse symmetric random matrix generated by MATLAB command \verb|sprandsym(n,d)| where $n$ is the matrix order and $d$ is the density, $-C$ is a well-conditioned diagonally dominant sparse $\SPD$ random matrix with elements normalized in the interval $[0,1]$ and with density $d$. These problems are denoted by \verb|RANDQEICP(d,n)|, where the density $d\in \{5\%,10\%,50\%,70\%,90\%\}$ and the order $n\in \{50,100,200,400,600\}$.
	
	Note that we only test on the equivalent formulation $\SEiCP{G,D}$ for \eqref{eq:sqeicp}. 
	Moreover, we convert $G$ to be SPD by setting $G = G + \mu_{GD} D$ where 
	$$\mu_{GD} = \min \{\mu: G+\mu D \succeq 0 \} + 1$$
	is solved by MOSEK. 
	
	\textbf{Setup:} The setups for the compared algorithms are summarized below
	\begin{itemize}[leftmargin=12pt]
		\item Initialization: we take random initial point $x^0\neq 0$ uniformly distributed in $[0,1]^n$ for \eqref{eq:seicp}. The initial point for \eqref{eq:sqeicp} is computed as follows:
		$x^0\neq 0$ is taken randomly as in \eqref{eq:seicp}, and $$\lambda_0 = \frac{-\langle x^0,Bx^0\rangle + \sqrt{\langle x^0,Bx^0\rangle^2 - 4\langle x^0,Ax^0\rangle\langle x^0,Cx^0\rangle}}{2\langle x^0,Ax^0\rangle},$$
		which is strictly positive since $A$ and $-C$ are $\SPD$ matrices. Then we get from $(2a)$ of \Cref{thm:SQEiCP=SEiCP} that $$z^0 = (1+\lambda_0)^{-1}(\lambda_0x^0,x^0),$$ which initializes $\SEiCP{G,D}$, and $\SEiCP{H,D}$ can be initialized in a similar way. Note that all compared methods use the same initial point for the fairness.  
		\item Termination criteria: For \eqref{eq:seicp}, DCA and BDCA are terminated if $$\|d^k\|_2/(1+\|z^k\|_2)\leq \varepsilon$$
		with $\varepsilon=10^{-6}$ for \Cref{alg:BDCA_for_QP} and $\varepsilon=10^{-8}$ for \Cref{alg:BDCA_for_LnP}. We also terminate DCA and BDCA when the number of iterations exceeds \verb|MaxIT|=$10000$. The compared solvers KNITRO and FILTERSD are terminated with their default settings. MATLAB FMINCON requires setting the parameter \verb|MaxFunEvals|$=10^{6}$ at least for finding most of feasible solutions for \eqref{eq:seicp}. FISTA is terminated if  $$\|u^{i+1}-u^i\|_2/(1+\|u^{i+1}\|_2)\leq 10^{-6}.$$     
		\item Other settings: For DCA and BDCA \Cref{alg:BDCA_for_QP}, the MOSEK 9.2 is applied to solve the convex subproblem \eqref{prob:QPk} using the default parameters. 
		For DCA and BDCA \Cref{alg:BDCA_for_LnP}, 
		we propose setting $\eta = L_g = n$ for the \eqref{prob:seicp_log} model of $\SEiCP{A,B}$ and $\eta = L_g = 2\max\{\kappa_G,\kappa_D\}$ for the \eqref{prob:seicp_log} model of $\SEiCP{G,D}$ instead of using the estimations in \Cref{lem:strongconvexityofgandh} and \Cref{cor:smoothnessofgandh}. These settings performed surprisingly well in our numerical tests. Note that when $\eta$ is large enough, then a smaller $\eta$ will lead to a well-conditioned subproblem \eqref{prob:LnPk} and a larger stepsize $1/L_g$ in FISTA, resulting better numerical performance in DCA and BDCA. 
		For FISTA,
		the parameter $L_i$ is picked using constant strategy, i.e., $L_i = L_g, \forall i$. The simplex projection is computed by \Cref{alg:Simplex_Proj}.
	\end{itemize}
	
	\textbf{Notations:}
	The following notations are used in the numerical results 
	\begin{itemize}[leftmargin=12pt]
		\item $\lambda$ - computed complementary eigenvalue;
		\item IT - number of iterations for DCA and BDCA;
		\item CPU - CPU time in seconds;
		\item $c$ - exponent of the value $10^{-c}$ of the feasibility measure of the computed solution, which is defined by 
		$$\|[x]_-\|_2 + \|[w]_-\|_2 + |w^{\top}x|,$$
		where $[x]_-$ is a vector defined by $[\min\{x_i,0\}]_{i=1}^n$, $w = \lambda Bx - Ax$ for \eqref{eq:seicp} and $w = \lambda^2 Ax + \lambda Bx + Cx$ for \eqref{eq:sqeicp}.
		\item avg - average results regarding to CPU, IT and $c$ for DCA and BDCA; CPU and $c$ for FMINCON, KNITRO and FILTERSD. 
	\end{itemize}
	Note that $(x,\lambda)$ should be considered as a solution of \eqref{eq:seicp} or \eqref{eq:sqeicp} if $c$ is big, i.e., $10^{-c}$ is small. The bigger $c$ is the better precisions of the eigenvalue and eigenvector are. 
	
	\subsection{Numerical results for \eqref{eq:seicp}}\label{subsec:resultsforseicp}
	\begin{table}[tbhp]
		\caption{Solutions of \eqref{eq:seicp} by DCA, BDCA, FMINCON, KNITRO and FILTERSD to the \eqref{prob:seicp_log} model on RANDEICP and NEP datasets.}
		\label{tab:performcomp_seicp_LnP}
		\centering
		\resizebox{\columnwidth}{!}{
			\begin{tabular}{l|cccc|cccc|ccc|ccc|ccc} \toprule
				\multirow{2}{*}{Prob} & \multicolumn{4}{c|}{DCA} & \multicolumn{4}{c|}{BDCA} & \multicolumn{3}{c|}{FMINCON} & \multicolumn{3}{c|}{KNITRO} & \multicolumn{3}{c}{FILTERSD} \\
				& $\lambda$ & CPU & IT & c & $\lambda$ & CPU & IT & c & $\lambda$ & CPU & c & $\lambda$ & CPU & c & $\lambda$ & CPU & c \\
				\midrule
				RANDEICP($-1,1,50$) & $3.9518$ & $0.007$ & $336$ & $6$ & $3.9518$ & $0.024$ & $89$ & $8$ & $3.9518$ & $0.451$ & $5$ & $3.9518$ & $0.145$ & $6$ & $3.9518$ & $0.015$ & $3$\\
				RANDEICP($-1,1,100$) & $5.5532$ & $0.013$ & $428$ & $6$ & $5.5532$ & $0.011$ & $97$ & $7$ & $5.5530$ & $0.365$ & $4$ & $5.5532$ & $0.082$ & $6$ & $5.5532$ & $0.026$ & $3$\\
				RANDEICP($-1,1,200$) & $8.3007$ & $0.038$ & $418$ & $6$ & $8.3007$ & $0.021$ & $106$ & $7$ & $8.3002$ & $1.826$ & $4$ & $8.3007$ & $0.593$ & $7$ & $8.3007$ & $0.172$ & $3$\\
				RANDEICP($-1,1,400$) & $11.8781$ & $0.176$ & $1571$ & $6$ & $11.8781$ & $0.128$ & $395$ & $7$ & $11.8768$ & $14.737$ & $4$ & $11.8781$ & $6.138$ & $7$ & $11.8780$ & $1.893$ & $3$\\
				RANDEICP($-1,1,600$) & $13.9468$ & $0.385$ & $1114$ & $6$ & $13.9468$ & $0.206$ & $274$ & $8$ & $13.9765$ & $62.308$ & $3$ & $13.9797$ & $26.038$ & $4$ & $13.9796$ & $4.331$ & $3$\\
				RANDEICP($-1,1,800$) & $16.3095$ & $1.841$ & $2765$ & $6$ & $16.3095$ & $0.868$ & $606$ & $8$ & $16.3044$ & $170.292$ & $3$ & $16.3095$ & $86.029$ & $6$ & $16.1180$ & $8.519$ & $4$\\
				\midrule
				RANDEICP($-10,10,50$) & $32.0015$ & $0.008$ & $604$ & $5$ & $32.0015$ & $0.004$ & $84$ & $6$ & $32.0009$ & $0.099$ & $3$ & $32.0015$ & $0.022$ & $5$ & $32.0010$ & $0.007$ & $2$\\
				RANDEICP($-10,10,100$) & $55.0963$ & $0.011$ & $361$ & $5$ & $55.0963$ & $0.009$ & $103$ & $7$ & $55.0945$ & $0.318$ & $3$ & $55.0963$ & $0.085$ & $5$ & $55.0961$ & $0.035$ & $2$\\
				RANDEICP($-10,10,200$) & $76.4938$ & $0.074$ & $1134$ & $5$ & $76.4938$ & $0.037$ & $251$ & $6$ & $76.0301$ & $1.697$ & $3$ & $76.4938$ & $0.545$ & $6$ & $76.4933$ & $0.089$ & $2$\\
				RANDEICP($-10,10,400$) & $113.2569$ & $0.138$ & $709$ & $5$ & $113.2569$ & $0.075$ & $187$ & $7$ & $113.2427$ & $13.326$ & $3$ & $113.2569$ & $5.737$ & $5$ & $113.2566$ & $1.420$ & $2$\\
				RANDEICP($-10,10,600$) & $141.0342$ & $0.309$ & $829$ & $5$ & $141.0342$ & $0.181$ & $218$ & $6$ & $141.0020$ & $48.917$ & $2$ & $141.0341$ & $24.668$ & $3$ & $141.0340$ & $3.870$ & $2$\\
				RANDEICP($-10,10,800$) & $159.6700$ & $4.621$ & $8902$ & $5$ & $159.6700$ & $1.925$ & $1743$ & $7$ & $159.6200$ & $215.517$ & $2$ & $159.6700$ & $80.110$ & $5$ & $159.6698$ & $12.154$ & $2$\\
				\midrule
				NEP-bfw398b & $-0.0000$ & $0.946$ & $10000$ & $6$ & $-0.0000$ & $0.105$ & $426$ & $7$ & $-0.0000$ & $11.575$ & $6$ & $-0.0000$ & $10.032$ & $7$ & $-0.0000$ & $0.055$ & $6$\\
				NEP-bfw782b & $-0.0000$ & $3.573$ & $10000$ & $6$ & $-0.0000$ & $0.239$ & $309$ & $6$ & $-0.0000$ & $48.161$ & $6$ & $-0.0000$ & $52.518$ & $7$ & $-0.0000$ & $0.170$ & $6$\\
				NEP-ck400 & $4.8132$ & $1.334$ & $10000$ & $2$ & $4.8187$ & $0.226$ & $908$ & $5$ & $4.8184$ & $6.663$ & $3$ & $4.8187$ & $0.974$ & $5$ & $4.7468$ & $0.124$ & $7$\\
				NEP-ck656 & $4.8048$ & $3.207$ & $10000$ & $2$ & $4.8187$ & $0.512$ & $916$ & $5$ & $1.0289$ & $7.265$ & $2$ & $4.8187$ & $3.621$ & $4$ & $4.7468$ & $0.156$ & $7$\\
				NEP-dwa512 & $0.7722$ & $2.128$ & $10000$ & $4$ & $0.7723$ & $1.246$ & $3074$ & $6$ & $0.7721$ & $16.891$ & $4$ & $0.7723$ & $14.738$ & $6$ & $0.7723$ & $3.303$ & $6$\\
				NEP-lop163 & $1.1087$ & $0.313$ & $10000$ & $5$ & $1.1087$ & $0.036$ & $415$ & $7$ & $1.1086$ & $0.605$ & $4$ & $1.1087$ & $0.353$ & $6$ & $1.1087$ & $0.133$ & $5$\\
				NEP-mhd416a & $1192.3774$ & $1.345$ & $10000$ & $0$ & $1192.8767$ & $0.077$ & $294$ & $4$ & $1192.7605$ & $3.487$ & $1$ & $1192.8767$ & $1.198$ & $2$ & $256.3736$ & $0.023$ & $0$\\
				NEP-olm100 & $213.6680$ & $0.080$ & $4604$ & $5$ & $213.6680$ & $0.052$ & $829$ & $5$ & $212.3586$ & $0.575$ & $0$ & $213.6680$ & $0.094$ & $4$ & $212.9640$ & $0.020$ & $0$\\
				NEP-olm500 & $5141.1850$ & $2.371$ & $10000$ & $0$ & $5143.3869$ & $4.042$ & $10000$ & $2$ & $4988.8881$ & $121.333$ & $-1$ & $5143.3880$ & $16.240$ & $2$ & $4559.2476$ & $0.057$ & $-2$\\
				NEP-rbs480a & $356.2392$ & $0.182$ & $683$ & $4$ & $356.2392$ & $0.104$ & $229$ & $7$ & $356.1817$ & $20.428$ & $2$ & $356.2391$ & $8.786$ & $2$ & $356.2388$ & $1.359$ & $1$\\
				NEP-rdb200 & $4.1363$ & $0.068$ & $1025$ & $6$ & $4.1363$ & $0.026$ & $183$ & $9$ & $4.1351$ & $1.133$ & $4$ & $4.1363$ & $0.488$ & $6$ & $4.1362$ & $0.101$ & $3$\\
				NEP-rdb800l & $4.3712$ & $2.500$ & $4142$ & $6$ & $4.3712$ & $0.510$ & $449$ & $7$ & $4.3667$ & $83.324$ & $3$ & $4.3621$ & $74.426$ & $3$ & $4.3711$ & $8.925$ & $3$\\
				NEP-tub100 & $6.2452$ & $0.175$ & $10000$ & $3$ & $6.2452$ & $0.075$ & $1230$ & $5$ & $6.2151$ & $0.402$ & $2$ & $6.2452$ & $0.089$ & $3$ & $-4.9911$ & $0.008$ & $0$\\
				\midrule
				avg & & $1.034$ & $4785$ & $5$ & & $0.430$ & $937$ & $6$ & & $34.068$ & $3$ & & $16.550$ & $5$ & & $1.879$ & $3$\\
				\bottomrule
		\end{tabular}}	
	\end{table}

	\begin{table}[tbhp]
		\caption{Solutions of \eqref{eq:seicp} by DCA, BDCA, FMINCON, KNITRO and FILTERSD to the \eqref{prob:seicp_qcqp} model on RANDEICP and NEP datasets.}
		\label{tab:performcomp_seicp_QP}
		\centering
			\resizebox{\columnwidth}{!}{
				\begin{tabular}{l|cccc|cccc|ccc|ccc|ccc} \toprule
					\multirow{2}{*}{Prob} & \multicolumn{4}{c|}{DCA} & \multicolumn{4}{c|}{BDCA} & \multicolumn{3}{c|}{FMINCON} & \multicolumn{3}{c|}{KNITRO} & \multicolumn{3}{c}{FILTERSD} \\
					& $\lambda$ & CPU & IT & c & $\lambda$ & CPU & IT & c & $\lambda$ & CPU & c & $\lambda$ & CPU & c & $\lambda$ & CPU & c \\
					\midrule
					RANDEICP($-1,1,50$) & $3.9518$ & $0.376$ & $267$ & $4$ & $3.9518$ & $0.313$ & $219$ & $4$ & $3.9518$ & $0.233$ & $5$ & $3.9518$ & $0.168$ & $5$ & $3.9518$ & $0.016$ & $4$\\
					RANDEICP($-1,1,100$) & $5.5532$ & $0.424$ & $237$ & $4$ & $5.5532$ & $0.410$ & $231$ & $4$ & $5.5532$ & $0.598$ & $4$ & $5.5532$ & $0.203$ & $5$ & $5.5532$ & $0.021$ & $4$\\
					RANDEICP($-1,1,200$) & $8.3007$ & $0.569$ & $257$ & $4$ & $8.3007$ & $0.499$ & $225$ & $4$ & $8.3006$ & $3.297$ & $4$ & $8.3007$ & $1.323$ & $5$ & $8.3007$ & $0.031$ & $4$\\
					RANDEICP($-1,1,400$) & $11.8781$ & $6.286$ & $1598$ & $4$ & $11.8781$ & $5.124$ & $1246$ & $4$ & $11.8780$ & $84.356$ & $4$ & $11.8781$ & $8.701$ & $4$ & $11.8781$ & $0.316$ & $4$\\
					RANDEICP($-1,1,600$) & $13.9797$ & $4.083$ & $693$ & $4$ & $13.9797$ & $3.849$ & $587$ & $4$ & $13.9793$ & $310.106$ & $3$ & $13.9797$ & $53.692$ & $4$ & $13.9797$ & $0.781$ & $4$\\
					RANDEICP($-1,1,800$) & $16.3095$ & $18.609$ & $2168$ & $4$ & $16.3095$ & $19.151$ & $2002$ & $4$ & $16.3090$ & $358.421$ & $3$ & $16.3095$ & $92.718$ & $4$ & $16.3095$ & $2.667$ & $4$\\
					\midrule
					RANDEICP($-10,10,50$) & $33.3194$ & $0.117$ & $81$ & $4$ & $33.3194$ & $0.114$ & $81$ & $4$ & $32.0015$ & $0.279$ & $4$ & $32.2528$ & $0.116$ & $4$ & $33.3194$ & $0.007$ & $5$\\
					RANDEICP($-10,10,100$) & $55.0963$ & $0.308$ & $177$ & $3$ & $55.0963$ & $0.252$ & $143$ & $3$ & $55.0963$ & $1.758$ & $4$ & $55.0963$ & $0.481$ & $2$ & $55.0963$ & $0.021$ & $4$\\
					RANDEICP($-10,10,200$) & $76.4938$ & $1.816$ & $808$ & $3$ & $76.4938$ & $1.784$ & $768$ & $3$ & $76.4937$ & $4.508$ & $3$ & $76.3689$ & $2.240$ & $2$ & $76.3689$ & $0.081$ & $4$\\
					RANDEICP($-10,10,400$) & $113.2569$ & $1.737$ & $462$ & $3$ & $113.2569$ & $1.860$ & $468$ & $3$ & $113.2568$ & $34.351$ & $4$ & $113.2569$ & $13.831$ & $3$ & $113.2569$ & $0.316$ & $4$\\
					RANDEICP($-10,10,600$) & $141.0342$ & $3.438$ & $610$ & $3$ & $141.0342$ & $3.022$ & $494$ & $3$ & $141.0338$ & $108.770$ & $3$ & $141.0342$ & $68.711$ & $2$ & $141.0342$ & $1.515$ & $4$\\
					RANDEICP($-10,10,800$) & $159.6700$ & $9.669$ & $1137$ & $3$ & $159.6700$ & $9.786$ & $1062$ & $3$ & $159.6699$ & $250.689$ & $4$ & $159.1671$ & $152.379$ & $3$ & $159.6700$ & $1.774$ & $4$\\
					\midrule
					NEP-bfw398b & $0.0000$ & $18.413$ & $5189$ & $5$ & $0.0000$ & $11.560$ & $3130$ & $5$ & $0.0000$ & $2.299$ & $5$ & $0.0000$ & $8.386$ & $5$ & $0.0000$ & $0.025$ & $5$\\
					NEP-bfw782b & $0.0000$ & $37.338$ & $5120$ & $5$ & $0.0000$ & $18.807$ & $2388$ & $5$ & $0.0000$ & $5.894$ & $5$ & $0.0000$ & $8.133$ & $5$ & $0.0000$ & $0.613$ & $7$\\
					NEP-ck400 & $4.8187$ & $2.308$ & $600$ & $6$ & $4.8187$ & $2.165$ & $540$ & $6$ & $4.8187$ & $5.515$ & $4$ & $4.8187$ & $1.146$ & $5$ & $4.7468$ & $0.225$ & $4$\\
					NEP-ck656 & $4.8187$ & $3.402$ & $535$ & $6$ & $4.8187$ & $3.429$ & $505$ & $6$ & $4.8186$ & $20.317$ & $4$ & $4.8187$ & $9.268$ & $5$ & $4.8187$ & $0.745$ & $4$\\
					NEP-dwa512 & $0.7723$ & $0.451$ & $97$ & $4$ & $0.7723$ & $0.385$ & $79$ & $4$ & $0.7722$ & $13.254$ & $4$ & $0.7723$ & $3.015$ & $4$ & $0.7723$ & $1.944$ & $6$\\
					NEP-lop163 & $1.1087$ & $0.637$ & $325$ & $4$ & $1.1087$ & $0.545$ & $272$ & $4$ & $1.1086$ & $0.522$ & $4$ & $1.1087$ & $0.100$ & $5$ & $1.1087$ & $0.037$ & $4$\\
					NEP-mhd416a & $1192.8766$ & $0.090$ & $22$ & $4$ & $1192.8766$ & $0.098$ & $22$ & $4$ & $1192.8766$ & $10.724$ & $3$ & $1192.8766$ & $2.398$ & $3$ & $1192.8767$ & $0.265$ & $3$\\
					NEP-olm100 & $213.6678$ & $5.366$ & $3118$ & $1$ & $213.6678$ & $5.097$ & $2942$ & $1$ & $213.6679$ & $3.730$ & $4$ & $213.6680$ & $1.179$ & $3$ & $213.6680$ & $0.077$ & $4$\\
					NEP-olm500 & $5143.0853$ & $47.110$ & $10000$ & $0$ & $5143.1075$ & $51.421$ & $10000$ & $0$ & $4779.1981$ & $37.677$ & $-2$ & $5143.3891$ & $130.078$ & $2$ & $5143.3891$ & $3.800$ & $2$\\
					NEP-rbs480a & $356.2392$ & $0.911$ & $206$ & $1$ & $356.2392$ & $0.858$ & $188$ & $1$ & $354.0079$ & $19.707$ & $0$ & $327.4180$ & $22.207$ & $3$ & $356.2392$ & $0.488$ & $4$\\
					NEP-rdb200 & $4.1363$ & $0.771$ & $359$ & $3$ & $4.1363$ & $0.759$ & $341$ & $3$ & $4.1363$ & $3.648$ & $4$ & $4.1363$ & $0.932$ & $4$ & $4.1363$ & $0.026$ & $4$\\
					NEP-rdb800l & $4.3712$ & $6.043$ & $789$ & $3$ & $4.3712$ & $5.516$ & $660$ & $3$ & $4.3177$ & $34.674$ & $1$ & $4.3712$ & $86.637$ & $5$ & $4.3712$ & $0.984$ & $4$\\
					NEP-tub100 & $6.2447$ & $4.402$ & $2417$ & $2$ & $6.2447$ & $3.691$ & $2017$ & $2$ & $6.2452$ & $1.074$ & $3$ & $6.2452$ & $0.158$ & $3$ & $6.2452$ & $0.057$ & $3$\\
					\midrule
					avg & & $6.987$ & $1491$ & $3$ & & $6.020$ & $1224$ & $3$ & & $52.656$ & $3$ & & $26.728$ & $4$ & & $0.673$ & $4$\\
					\bottomrule
			\end{tabular}}	
	\end{table}

	The numerical results in \Cref{tab:performcomp_seicp_QP,tab:performcomp_seicp_LnP} for \eqref{prob:seicp_log} and \eqref{prob:seicp_qcqp} models on both RANDEICP and NEP datasets lead to the following observations:
	\begin{itemize}[leftmargin=12pt]
		\item For \eqref{prob:seicp_log} model, the best numerical results are always obtained by BDCA with $0.43$ seconds in average CPU time and with best quality of computed solutions (with the largest average exponent $c=6$); whereas for \eqref{prob:seicp_qcqp} model, the best numerical results are almost always obtained by FILTERSD with the minimal average CPU time $0.673$ seconds and with the largest average exponent $c=4$. The second winner for \eqref{prob:seicp_log} model is DCA, then follows by KNITRO, FILTERSD and FMINCON;  whereas for \eqref{prob:seicp_qcqp} model, the second winner is BDCA, then follows by DCA, KNITRO and FMINCON. Note that FMINCON always obtains the worst numerical results both in average CPU time and in solution quality for solving \eqref{prob:seicp_log} and \eqref{prob:seicp_qcqp}. Note that, the method with best numerical performance among all compared algorithms is BDCA for solving \eqref{prob:seicp_qcqp} model.
		\item BDCA outperforms DCA with about $80\%$ (resp. $18\%$) reduction in the average number of iterations and about $58\%$ (resp. $14\%$) reduction in average CPU time for solving \eqref{prob:seicp_log} (resp. \eqref{prob:seicp_qcqp}) model. Hence, BDCA yields better acceleration to the \eqref{prob:seicp_log} model than the \eqref{prob:seicp_qcqp} model. Moreover, the quality of the computed solution is also better in BDCA than in DCA. 
		\item Moreover, in some instances of the NEP dataset, the number of iterations for DCA and BDCA exceed the threshold for maximum number of iterations $10000$ (particularly in DCA for \eqref{prob:seicp_log}), however the quality of the computed results seems still good enough with $c> 3$ in average for these instances. Furthermore, FILTERSD and FMINCON may fail to solve some ill-conditioned instances of the NEP dataset (e.g., \verb|NEP-olm500| with $\kappa_A=2.3\times 10^4$, \verb|NEP-mhd416a| with $\kappa_A=2.5\times 10^3$ and \verb|NEP-tub100| with $\kappa_A=1.9\times 10^3$), while BDCA and DCA successfully solved all test problems. 
	\end{itemize}
	
	\subsection{Numerical results for \eqref{eq:sqeicp}}\label{subsec:resultsforsqeicp}
	\begin{table}[tbhp]
		\caption{Solutions of \eqref{eq:sqeicp} by DCA, BDCA, FMINCON, KNITRO and FILTERSD to the \eqref{prob:seicp_log} model of the $\SEiCP{G,D}$ formulation on the RANDQEICP dataset.}
		\label{tab:performcomp_sqeicp_LnP}
		\centering
		\resizebox{\columnwidth}{!}{
			\begin{tabular}{l|cccc|cccc|ccc|ccc|ccc} \toprule
				\multirow{2}{*}{Prob} & \multicolumn{4}{c|}{DCA} & \multicolumn{4}{c|}{BDCA} & \multicolumn{3}{c|}{FMINCON} & \multicolumn{3}{c|}{KNITRO} & \multicolumn{3}{c}{FILTERSD} \\
				& $\lambda$ & CPU & IT & c & $\lambda$ & CPU & IT & c & $\lambda$ & CPU & c & $\lambda$ & CPU & c & $\lambda$ & CPU & c \\
				\midrule
				RANDQEICP($5\%,50$) & $-1.3091$ & $0.116$ & $3756$ & $8$ & $-1.3091$ & $0.027$ & $365$ & $6$ & $-1.3092$ & $0.275$ & $3$ & $-1.3091$ & $0.036$ & $5$ & $-1.3091$ & $0.026$ & $3$\\
				RANDQEICP($5\%,100$) & $-1.5469$ & $0.137$ & $8334$ & $5$ & $-1.5469$ & $0.026$ & $388$ & $5$ & $-1.5471$ & $0.874$ & $2$ & $-1.5469$ & $0.304$ & $4$ & $-2.2533$ & $0.008$ & $3$\\
				RANDQEICP($5\%,200$) & $-1.9266$ & $0.376$ & $10000$ & $4$ & $-1.9266$ & $0.108$ & $822$ & $5$ & $-1.9270$ & $3.108$ & $2$ & $-1.9266$ & $0.371$ & $4$ & $-2.1274$ & $0.035$ & $2$\\
				RANDQEICP($5\%,400$) & $-2.0098$ & $1.180$ & $10000$ & $3$ & $-2.0098$ & $0.247$ & $1096$ & $5$ & $-2.0107$ & $27.752$ & $1$ & $-2.0098$ & $3.161$ & $3$ & $-2.0098$ & $0.429$ & $3$\\
				RANDQEICP($5\%,600$) & $-1.7119$ & $1.325$ & $9173$ & $5$ & $-1.7119$ & $0.273$ & $830$ & $5$ & $-1.7134$ & $81.489$ & $1$ & $-1.7119$ & $10.806$ & $2$ & $-1.7119$ & $0.979$ & $3$\\
				\midrule
				RANDQEICP($10\%,50$) & $-2.4761$ & $0.184$ & $10000$ & $3$ & $-2.4761$ & $0.209$ & $4475$ & $6$ & $-2.4762$ & $0.190$ & $2$ & $-2.4761$ & $0.038$ & $4$ & $-2.4821$ & $0.016$ & $5$\\
				RANDQEICP($10\%,100$) & $-1.3290$ & $0.158$ & $8088$ & $6$ & $-1.4838$ & $0.013$ & $169$ & $6$ & $-1.3291$ & $0.629$ & $3$ & $-1.3290$ & $0.094$ & $5$ & $-1.4838$ & $0.008$ & $3$\\
				RANDQEICP($10\%,200$) & $-1.6070$ & $0.251$ & $7830$ & $5$ & $-1.6070$ & $0.063$ & $632$ & $5$ & $-1.6074$ & $3.316$ & $2$ & $-1.6070$ & $0.387$ & $3$ & $-1.6071$ & $0.014$ & $2$\\
				RANDQEICP($10\%,400$) & $-1.5442$ & $0.422$ & $3570$ & $5$ & $-1.5442$ & $0.117$ & $518$ & $5$ & $-1.5451$ & $18.145$ & $1$ & $-1.5442$ & $3.243$ & $3$ & $-1.7060$ & $0.182$ & $2$\\
				RANDQEICP($10\%,600$) & $-2.0771$ & $0.487$ & $2955$ & $4$ & $-2.0771$ & $0.219$ & $590$ & $5$ & $-2.0785$ & $66.369$ & $1$ & $-2.0771$ & $13.738$ & $2$ & $-2.0771$ & $0.842$ & $3$\\
				\midrule
				RANDQEICP($50\%,50$) & $-1.0452$ & $0.037$ & $2346$ & $6$ & $-1.0452$ & $0.004$ & $45$ & $6$ & $-1.0453$ & $0.199$ & $3$ & $-1.0452$ & $0.033$ & $5$ & $-1.8949$ & $0.005$ & $3$\\
				RANDQEICP($50\%,100$) & $-1.0273$ & $0.131$ & $7520$ & $6$ & $-1.0273$ & $0.019$ & $277$ & $7$ & $-1.0275$ & $0.491$ & $2$ & $-1.0273$ & $0.060$ & $4$ & $-1.0273$ & $0.005$ & $3$\\
				RANDQEICP($50\%,200$) & $-1.6939$ & $0.257$ & $8143$ & $5$ & $-1.6939$ & $0.043$ & $402$ & $5$ & $-1.6943$ & $2.855$ & $1$ & $-1.6939$ & $0.487$ & $3$ & $-2.8739$ & $0.037$ & $3$\\
				RANDQEICP($50\%,400$) & $-1.7567$ & $0.496$ & $3594$ & $5$ & $-1.7567$ & $0.159$ & $475$ & $5$ & $-1.7575$ & $20.022$ & $1$ & $-1.7567$ & $3.639$ & $3$ & $-1.7567$ & $0.303$ & $3$\\
				RANDQEICP($50\%,600$) & $-2.3685$ & $0.886$ & $6102$ & $5$ & $-2.3685$ & $0.208$ & $640$ & $5$ & $-2.3700$ & $53.110$ & $1$ & $-2.3685$ & $12.892$ & $3$ & $-2.3685$ & $0.864$ & $3$\\
				\midrule
				RANDQEICP($70\%,50$) & $-1.6623$ & $0.079$ & $5021$ & $7$ & $-1.6623$ & $0.022$ & $429$ & $6$ & $-1.6623$ & $0.209$ & $3$ & $-1.6623$ & $0.035$ & $5$ & $-1.9925$ & $0.003$ & $3$\\
				RANDQEICP($70\%,100$) & $-1.0026$ & $0.132$ & $7481$ & $6$ & $-1.0026$ & $0.005$ & $40$ & $6$ & $-1.0027$ & $0.461$ & $2$ & $-1.0026$ & $0.250$ & $4$ & $-2.1216$ & $0.006$ & $3$\\
				RANDQEICP($70\%,200$) & $-1.2494$ & $0.279$ & $8660$ & $5$ & $-1.2494$ & $0.046$ & $431$ & $5$ & $-1.2499$ & $2.867$ & $2$ & $-1.2494$ & $0.410$ & $4$ & $-2.9749$ & $0.011$ & $2$\\
				RANDQEICP($70\%,400$) & $-1.4945$ & $0.992$ & $10000$ & $3$ & $-1.4945$ & $0.169$ & $860$ & $4$ & $-1.4954$ & $13.126$ & $1$ & $-1.4945$ & $1.962$ & $2$ & $-1.4945$ & $0.232$ & $2$\\
				RANDQEICP($70\%,600$) & $-2.2780$ & $1.654$ & $10000$ & $4$ & $-2.2780$ & $0.449$ & $1265$ & $5$ & $-2.2792$ & $53.652$ & $1$ & $-2.2778$ & $10.440$ & $2$ & $-2.4112$ & $0.760$ & $3$\\
				\midrule
				RANDQEICP($90\%,50$) & $-1.9971$ & $0.102$ & $6365$ & $10$ & $-1.9971$ & $0.024$ & $451$ & $10$ & $-1.9971$ & $0.200$ & $3$ & $-1.9971$ & $0.044$ & $6$ & $-2.1125$ & $0.003$ & $3$\\
				RANDQEICP($90\%,100$) & $-1.7783$ & $0.093$ & $3841$ & $6$ & $-1.7783$ & $0.007$ & $84$ & $6$ & $-1.7784$ & $0.606$ & $2$ & $-1.7783$ & $0.091$ & $5$ & $-2.8556$ & $0.015$ & $3$\\
				RANDQEICP($90\%,200$) & $-1.8497$ & $0.356$ & $10000$ & $3$ & $-1.8497$ & $0.170$ & $2286$ & $5$ & $-1.8557$ & $2.934$ & $1$ & $-1.8497$ & $0.418$ & $3$ & $-3.2654$ & $0.010$ & $3$\\
				RANDQEICP($90\%,400$) & $-1.8730$ & $0.426$ & $3295$ & $5$ & $-1.8730$ & $0.111$ & $401$ & $6$ & $-1.8739$ & $19.103$ & $1$ & $-1.8730$ & $5.609$ & $3$ & $-2.2368$ & $0.352$ & $3$\\
				RANDQEICP($90\%,600$) & $-1.9274$ & $0.567$ & $2986$ & $5$ & $-1.9274$ & $0.169$ & $426$ & $5$ & $-1.9347$ & $59.225$ & $1$ & $-1.9274$ & $12.907$ & $3$ & $-1.9274$ & $0.700$ & $2$\\
				\midrule
				avg & & $0.445$ & $6762$ & $5$ & & $0.116$ & $736$ & $6$ & & $17.248$ & $2$ & & $3.258$ & $4$ & & $0.234$ & $3$\\
				\bottomrule
		\end{tabular}}	
	\end{table}

	\begin{table}[tbhp]
		\caption{Solutions of \eqref{eq:sqeicp} by DCA, BDCA, FMINCON, KNITRO and FILTERSD to the \eqref{prob:seicp_qcqp} model of the $\SEiCP{G,D}$ formulation on the RANDQEICP dataset.}
		\label{tab:performcomp_sqeicp_QP}
		\centering
		\resizebox{\columnwidth}{!}{
			\begin{tabular}{l|cccc|cccc|ccc|ccc|ccc} \toprule
				\multirow{2}{*}{Prob} & \multicolumn{4}{c|}{DCA} & \multicolumn{4}{c|}{BDCA} & \multicolumn{3}{c|}{FMINCON} & \multicolumn{3}{c|}{KNITRO} & \multicolumn{3}{c}{FILTERSD} \\
				& $\lambda$ & CPU & IT & c & $\lambda$ & CPU & IT & c & $\lambda$ & CPU & c & $\lambda$ & CPU & c & $\lambda$ & CPU & c \\
				\midrule
				RANDQEICP($5\%,50$) & $-1.3091$ & $0.234$ & $99$ & $5$ & $-1.3091$ & $0.278$ & $118$ & $6$ & $-1.3092$ & $0.433$ & $2$ & $-1.3091$ & $0.066$ & $4$ & $-1.3091$ & $0.008$ & $3$\\
				RANDQEICP($5\%,100$) & $-1.5469$ & $0.357$ & $111$ & $3$ & $-1.5469$ & $0.605$ & $177$ & $3$ & $-1.5471$ & $0.695$ & $1$ & $-1.5469$ & $0.101$ & $3$ & $-1.5469$ & $0.016$ & $3$\\
				RANDQEICP($5\%,200$) & $-1.9266$ & $3.353$ & $452$ & $3$ & $-1.9266$ & $2.870$ & $383$ & $3$ & $-1.9267$ & $5.073$ & $1$ & $-1.9266$ & $0.834$ & $3$ & $-1.9266$ & $0.050$ & $3$\\
				RANDQEICP($5\%,400$) & $-2.0098$ & $6.729$ & $447$ & $2$ & $-2.0098$ & $6.356$ & $419$ & $2$ & $-2.0099$ & $36.612$ & $1$ & $-2.0098$ & $8.246$ & $2$ & $-2.0098$ & $0.214$ & $5$\\
				RANDQEICP($5\%,600$) & $-1.7119$ & $7.411$ & $241$ & $2$ & $-1.7119$ & $6.318$ & $204$ & $2$ & $-1.7126$ & $103.580$ & $0$ & $-1.7120$ & $42.192$ & $0$ & $-1.7119$ & $0.501$ & $4$\\
				\midrule
				RANDQEICP($10\%,50$) & $-2.4761$ & $14.695$ & $6694$ & $3$ & $-2.4761$ & $10.461$ & $4702$ & $3$ & $-2.4761$ & $0.276$ & $2$ & $-2.4761$ & $0.141$ & $4$ & $-2.4821$ & $0.006$ & $4$\\
				RANDQEICP($10\%,100$) & $-1.3290$ & $1.198$ & $379$ & $5$ & $-1.3290$ & $0.881$ & $282$ & $5$ & $-1.3291$ & $0.741$ & $2$ & $-1.3290$ & $0.100$ & $4$ & $-1.3290$ & $0.012$ & $8$\\
				RANDQEICP($10\%,200$) & $-1.5970$ & $6.778$ & $904$ & $2$ & $-1.5970$ & $7.012$ & $894$ & $2$ & $-1.5970$ & $6.833$ & $1$ & $-1.5970$ & $1.348$ & $2$ & $-1.5970$ & $0.077$ & $5$\\
				RANDQEICP($10\%,400$) & $-1.5442$ & $6.556$ & $422$ & $2$ & $-1.5442$ & $5.848$ & $378$ & $2$ & $-1.5443$ & $38.751$ & $1$ & $-1.5442$ & $9.685$ & $2$ & $-1.5442$ & $0.234$ & $5$\\
				RANDQEICP($10\%,600$) & $-1.8325$ & $8.441$ & $269$ & $2$ & $-1.8325$ & $9.502$ & $304$ & $2$ & $-1.8332$ & $114.343$ & $0$ & $-1.8325$ & $44.799$ & $1$ & $-1.8325$ & $0.492$ & $5$\\
				\midrule
				RANDQEICP($50\%,50$) & $-1.0452$ & $0.155$ & $74$ & $3$ & $-1.0452$ & $0.257$ & $119$ & $4$ & $-1.0453$ & $0.197$ & $2$ & $-1.0452$ & $0.036$ & $4$ & $-1.0452$ & $0.006$ & $3$\\
				RANDQEICP($50\%,100$) & $-1.0273$ & $0.348$ & $100$ & $3$ & $-1.0273$ & $0.380$ & $108$ & $3$ & $-1.0274$ & $0.590$ & $2$ & $-1.0273$ & $0.092$ & $4$ & $-1.0273$ & $0.014$ & $3$\\
				RANDQEICP($50\%,200$) & $-1.6939$ & $1.580$ & $207$ & $2$ & $-1.6939$ & $1.464$ & $184$ & $2$ & $-1.6942$ & $4.355$ & $1$ & $-1.6939$ & $0.787$ & $3$ & $-1.6939$ & $0.071$ & $5$\\
				RANDQEICP($50\%,400$) & $-1.7567$ & $5.555$ & $359$ & $2$ & $-1.7567$ & $4.866$ & $312$ & $2$ & $-1.7568$ & $42.229$ & $1$ & $-1.7567$ & $12.943$ & $2$ & $-1.7567$ & $0.251$ & $5$\\
				RANDQEICP($50\%,600$) & $-2.3685$ & $26.046$ & $847$ & $2$ & $-2.3685$ & $20.813$ & $673$ & $2$ & $-2.3686$ & $124.180$ & $0$ & $-2.3685$ & $81.686$ & $1$ & $-2.3685$ & $0.502$ & $3$\\
				\midrule
				RANDQEICP($70\%,50$) & $-1.6623$ & $0.364$ & $172$ & $5$ & $-1.6623$ & $0.357$ & $161$ & $5$ & $-1.6623$ & $0.214$ & $3$ & $-1.6623$ & $0.049$ & $4$ & $-1.6623$ & $0.006$ & $3$\\
				RANDQEICP($70\%,100$) & $-1.0026$ & $0.285$ & $90$ & $4$ & $-1.0026$ & $0.327$ & $97$ & $4$ & $-1.0027$ & $0.604$ & $2$ & $-1.0026$ & $0.132$ & $4$ & $-1.0026$ & $0.013$ & $4$\\
				RANDQEICP($70\%,200$) & $-1.2494$ & $1.056$ & $148$ & $3$ & $-1.2494$ & $1.087$ & $147$ & $3$ & $-1.2497$ & $3.969$ & $1$ & $-1.2494$ & $0.538$ & $2$ & $-1.2494$ & $0.054$ & $5$\\
				RANDQEICP($70\%,400$) & $-1.4945$ & $1.469$ & $98$ & $2$ & $-1.4945$ & $1.346$ & $88$ & $2$ & $-1.4945$ & $29.390$ & $1$ & $-1.4945$ & $6.711$ & $2$ & $-1.4945$ & $0.192$ & $4$\\
				RANDQEICP($70\%,600$) & $-2.2778$ & $27.557$ & $913$ & $2$ & $-2.2778$ & $27.701$ & $910$ & $2$ & $-2.2780$ & $121.946$ & $1$ & $-2.2780$ & $68.905$ & $1$ & $-2.4112$ & $0.552$ & $5$\\
				\midrule
				RANDQEICP($90\%,50$) & $-1.9971$ & $1.064$ & $493$ & $5$ & $-1.9971$ & $0.966$ & $439$ & $5$ & $-1.9971$ & $0.204$ & $3$ & $-1.9971$ & $0.065$ & $4$ & $-1.9971$ & $0.011$ & $4$\\
				RANDQEICP($90\%,100$) & $-1.7783$ & $1.604$ & $482$ & $3$ & $-1.7783$ & $1.243$ & $378$ & $3$ & $-1.7784$ & $0.707$ & $2$ & $-1.7783$ & $0.112$ & $4$ & $-1.7783$ & $0.017$ & $6$\\
				RANDQEICP($90\%,200$) & $-1.8497$ & $45.023$ & $6126$ & $2$ & $-1.8497$ & $27.906$ & $3775$ & $2$ & $-1.8499$ & $5.604$ & $1$ & $-1.8497$ & $0.808$ & $3$ & $-1.8497$ & $0.054$ & $5$\\
				RANDQEICP($90\%,400$) & $-1.8730$ & $2.615$ & $173$ & $2$ & $-1.8730$ & $2.545$ & $170$ & $2$ & $-1.8731$ & $41.465$ & $1$ & $-1.8730$ & $12.281$ & $2$ & $-1.8730$ & $0.208$ & $5$\\
				RANDQEICP($90\%,600$) & $-1.9274$ & $8.256$ & $270$ & $2$ & $-1.9274$ & $8.751$ & $285$ & $2$ & $-1.9275$ & $127.852$ & $0$ & $-1.9274$ & $97.187$ & $1$ & $-1.9274$ & $0.515$ & $4$\\
				\midrule
				avg & & $7.149$ & $823$ & $3$ & & $6.006$ & $628$ & $3$ & & $32.434$ & $1$ & & $15.594$ & $3$ & & $0.163$ & $4$\\
				\bottomrule
		\end{tabular}}	
	\end{table}

	The numerical results in \Cref{tab:performcomp_sqeicp_LnP,tab:performcomp_sqeicp_QP} for \eqref{prob:seicp_log} and \eqref{prob:seicp_qcqp} models to $\SEiCP{G,D}$ formulation on RANDQEICP dataset lead to similar observations as in \Cref{subsec:resultsforseicp} for \eqref{eq:seicp}. The negative value in $\lambda$ is because we subtract $\mu_{GD}$ from the computed $\lambda$ for $\SEiCP{G+\mu_{GD} D,D}$ according to \Cref{thm:1} to get $\lambda$ for $\SEiCP{G,D}$. The best average result is always obtained by BDCA for \eqref{prob:seicp_log} model with average CPU time $0.116$ seconds and with best average precision $c=6$, whereas the worst average result is always given by FMINCON in terms of the average CPU time and average precision for both \eqref{prob:seicp_log} and \eqref{prob:seicp_qcqp} models. BDCA outperformed DCA with better precision in numerical results and with about $89\%$ (resp. $24\%$) reduction in the average number of iterations and about $74\%$ (resp. $16\%$) reduction in average CPU time for solving \eqref{prob:seicp_log} (resp. \eqref{prob:seicp_qcqp}) model. Hence, BDCA yields better acceleration to the \eqref{prob:seicp_log} model than the \eqref{prob:seicp_qcqp} model. 
	
	We can conclude that BDCA significantly speeds up the convergence of DCA to get better numerical results, and often outperforms other compared solvers. Hence, BDCA should be a promising approach for solving \eqref{eq:seicp} and \eqref{eq:sqeicp}, especially for large-scale cases.  
		
	\section{Conclusions}
	\label{sec:conclusions}
	
	In this paper, we presented two DC programming formulations and the corresponding accelerated DC algorithms (BDCA) for solving \eqref{eq:seicp} and \eqref{eq:sqeicp}. Numerical simulations of BDCA and DCA against KNITRO, FILTERSD and MATLAB FMINCON solvers, and tested on both synthetic datasets and Matrix Market NEP Repository for \eqref{eq:seicp} and \eqref{eq:sqeicp}, demonstrated that BDCA accelerates dramatically the convergence of DCA to get better numerical solutions, and often outperforms the compared solvers (KNITRO, FILTERSD and FMINCON) in terms of the average CPU time and average solution precision. BDCA is a promising approach for solving both \eqref{eq:seicp} and \eqref{eq:sqeicp}, especially for large-scale cases.

	There are several questions that deserve attention in the future: $(i)$ Apply BDCA to solve asymmetric EiCP (AEiCP) and asymmetric QEiCP (AQEiCP). As opposed to the symmetric cases, the formulations \eqref{prob:seicp_qcqp} and \eqref{prob:seicp_log} are no longer equivalent to AEiCP anymore. We have to consider some nonlinear programming formulations (NLP) such as those proposed in \cite{Niu12,niu2019improved}, and investigate the corresponding BDCA. How to efficiently solve the convex subproblems and how to proceed inexpensive exact line search will be two important questions to study. $(ii)$ Propose a better solution approach for the convex subproblem \eqref{prob:QPk} without using any external solver. The problem \eqref{prob:QPk} has a very special structure with only one convex quadratic constraint and nonnegative orthant by minimizing a linear objective function, so we believe that by ingeniously exploiting the unique structure, it could be solved either explicitly or more efficiently than invoking external solvers. $(iii)$ Estimate smaller $\bar{\mu}$ and $L_g$ for the \eqref{prob:seicp_log} model. As observed in our numerical tests, the estimations in \Cref{lem:strongconvexityofgandh} and \Cref{cor:smoothnessofgandh} are highly overestimated. Smaller parameters performed much better in practice. A possible idea is to develop an efficient adaptive procedure for $\mu$ (perhaps similar to the one proposed for $L_i$ in FISTA), which does not aim to ensure a global convexity of $g$ and $h$ over $\Omega$, but to guarantee a local convexity of $g$ around some convex neighborhoods of the current iterate $x^k$ containing the next iterate $x^{k+1}$, leading to better local convex subproblems of the DC program than the global ones leveraged in this paper. We may call this new algorithm as \emph{Quasi-DCA}, whose convergence analysis, accelerated variants and numerical performance in various challenging applications deserve more attention in the future.  
	
	\section*{Acknowledgments}
	This work was funded by the Natural Science Foundation of China (Grant No: 11601327). Special thanks to Professor Joaquim J. Judice for his kind encouragement and stimulating discussions on several aspects of this paper. 
	
	\bibliographystyle{siamplain}
	\bibliography{references}
\end{document}